\documentclass[11pt]{article}

\title{Janossy densities for Unitary ensembles at the spectral edge}
\date{\today} 
\author{Brian Rider \\
Department of Mathematics \\ University of Colorado at Boulder 
\\  Boulder, CO 80309
 \and Xin Zhou \\
Department of Mathematics \\  Duke University \\
 Durham, NC 27708
}

\usepackage{amsfonts}
\usepackage{graphicx}
\usepackage{amsmath}
\usepackage{amsthm}
\usepackage{amssymb}
\usepackage{stmaryrd}
\usepackage{epsfig}

\newcounter{num}
    \newtheorem{theorem}{Theorem}[section]
    \newtheorem{lemma}[theorem]{Lemma}
     \newtheorem{proposition}[theorem]{Proposition}
         \newtheorem{cor}[theorem]{Corollary}
       \newtheorem{claim}[theorem]{Claim}
      \newtheorem{definition}[theorem]{Definition}
     \newtheorem{remark}[theorem]{Remark}

\newcommand{\beq}        {\begin{eqnarray}}
\newcommand{\eeq}        {\end{eqnarray}}
\newcommand{\be}         {\begin{equation}}
\newcommand{\ee}         {\end{equation}}
\newcommand{\beqn}       {\begin{eqnarray*}}
\newcommand{\eeqn}       {\end{eqnarray*}}
\newcommand{\ba}         {\begin{array}}
\newcommand{\ea}         {\end{array}}

\newcommand\ep{\varepsilon}
\newcommand\R{\mathbb R}
\newcommand\C{\mathbb C}
\newcommand\al{{\alpha}}

\newcommand\tht{{\theta}}

\newcommand\wt{\widetilde}
\newcommand\wh{\widehat}

\newcommand\tr{\rm{ tr }}

\newcommand\ra{{\rightarrow}}
\newcommand\la{{\leftarrow}}

\newcommand\lra{{\stackrel{\shortrightarrow}{ \shortleftarrow} }}

\newcommand\pr{{\prime}}

\newcommand\supp{\rm{{supp}}}
\newcommand\ID{{\mathbb I}}
\newcommand{\one}{{\mathbf 1}}

\newcommand\Ai{\rm{ Ai}}
\newcommand\om{\omega}

\newcommand\RE{{\rm Re} \,}
\newcommand\IM{{\rm Im} \,}
\newcommand\sra{{\shortrightarrow}}
\newcommand\sla{{\shortleftarrow}}
\newcommand\OO{{\mathcal O}}

\begin{document}
\maketitle

\begin{abstract}   For a broad class of unitary ensembles of random matrices we demonstrate
the universal nature of the Janossy densities of eigenvalues near the spectral edge,
providing  a different formulation of the probability distributions of the limiting
second, third, etc. largest eigenvalues
of the ensembles in question. The
approach is based on a representation of the Janossy densities in terms of a system of
orthogonal polynomials,  plus the steepest
descent method of Deift and Zhou for the asymptotic analysis of the associated Riemann-Hilbert
problem.
\end{abstract}







\section{Introduction}
\setcounter{equation}{0}
\label{sec:int}

Consider the probability measure ${ P}_n$ on the space of $n \times n$ Hermitian matrices $M$
defined by
\[
\label{equ1}
d { P}_n(M) = \frac{1}{Z_n} e^{- n \, \tr V(M)}  \,  d M,
\]
in which $\tr$ denotes the matrix trace,
$dM$ is the Lebesgue measure,
and the potential $V$ grows sufficiently fast at $\pm \infty$
so that the normalizer $Z_n < \infty$.   This prescription is an instance of the
{\em unitary ensembles} of Random Matrix Theory;  the invariance
$d {P}_n(U^*MU) = d {P}_n(M)$ for any $n \times n$ unitary
matrix $U$  explains the terminology.

Regarding their spectral properties 
these ensembles are integrable.  That is to say, the  joint
probability density of the eigenvalues $x_1,  x_2, \dots, x_n$ induced by $P_n$
may be computed:
\be
\label{e1}
 {\rho}_n(x_1,  \dots, x_n)  =  \frac{1}{\hat{Z}_n} \prod_{1 \le \ell  < k \le n}  | x_{\ell} - x_k |^2
                                   e^{- n \sum_{k=1}^n V (x_k) },
\ee
with a new normalizer ${\hat Z}_n$.  Even more, all finite dimensional correlation
functions of the eigenvalues,
\be
\label{e2}
{\rho}_{n}^{(k)} (x_1, \dots, x_k)   \equiv    \frac{n !}{(n-k)!} \,
                                                \int_{-\infty}^{\infty} \cdots \int_{-\infty}^{\infty}
                                             {\rho}_n(x_1, \dots,x_k, {\bar x}_{k+1}, \dots, {\bar x}_n)  \,
                                             d {\bar x}_{k+1}  \cdots d {\bar x}_n,
\ee
have explicit expressions.  Bring in the system of polynomials
\[
p_{k,n}(x)  = \gamma_{k,n} x^k + \dots,
\]
$k = 1, \dots, n$ with $\gamma_{k,n} > 0$,  orthonormal with respect to the weight $w_n(x) \equiv e^{-n V(x)}$
over $\R$.  That is,
$
     \int_{-\infty}^{\infty}  p_{\ell,n}(x) p_{k,n}(x) w_n(x) dx   =  \delta_{\ell k}$, and it holds
\be
\label{cor}
    {\rho}_n^{(k)}(x_1, \dots, x_k)   =  \det \Bigl[  K_n(x_{\ell}, x_m) \Bigr] _{1 \le \ell, m \le  k},
\ee
in which
\beq
\label{CD}
   K_n(x, y)  & =  &   \sqrt{w_n(x)} \sqrt{w_n(y)} \,  \sum_{k=0}^{n-1}  p_{k,n}(x) p_{k,n}(y)  \\
                     & = &  \sqrt{w_n(x)} \sqrt{w_n(y)} \,   \frac{\gamma_{n-1, n}}{\gamma_{n,n} } \
                          \frac{ p_{n,n}(x)  p_{n-1,n}(y)   -  p_{n-1,n}(x) p_{n,n}(y)}{x-y}, \nonumber
\eeq
by the formula of Christoffel-Darboux.  The form of (\ref{cor}) implies that the ensemble eigenvalues comprise
a {\em determinantal}  point process.

With the above normalization, $\rho_n^{(k)}(x_1, \dots, x_k)$ is
really a joint intensity of there being an eigenvalue, irrespective of order, at each of the points
$x_1$ through $x_k$.  Alternatively, fix a subset ${\Gamma}$ of ${\R}$ containing
$x_1, \dots, x_k$.  Then, the probability that there are exactly $k$ eigenvalues in ${\Gamma}$, one at each
of those same points, defines the $k$-th level Janossy density, denoted   by
${\cal J}_{n, {\Gamma}}^{(k)}(x_1, \dots, x_k)$.  For any determinantal point processes
the   Janossy densities are also determinantal  (\cite{DJ88} p. 140): in our case,
\be
\label{Janossy}
   {\cal J}_{n, {\Gamma}}^{(k)}(x_1, \dots, x_k) =
    D( {\Gamma} ) \times \det \Bigl[  L_{n, {\Gamma}} (x_{\ell}, x_m) \Bigr]_{1 \le \ell, m \le  k},
\ee
where
\be
\label{LK}
L_{n, {\Gamma}}  = K_{n, {\Gamma}} ( \ID - K_{n, {\Gamma}} )^{-1},
\ee
the kernel $K_{n, {\Gamma}}(x,y)$ equaling  ${\one}_{\Gamma}(x) K_n(x,y)  {\one}_{\Gamma}(y)$, 
and the prefactor  $D(\Gamma)$
is the Fredholm determinant
\be
   D(\Gamma) = \det ( \ID  - K_{n, {\Gamma}} ).
\ee
More important for what follows, it has recently been shown in
\cite{BS03} that  kernel of $L_{n, \Gamma}$ is also
Christoffel-Darboux  type. In particular,
\be
\label{Lker}
     L_{n, \Gamma}(x,y)  =  \sqrt{w_n(x)} \sqrt{w_n(y)} \,   \frac{{\tilde  \gamma_{n-1, n}}}{{\tilde \gamma_{n,n}} } \
                          \frac{ {\tilde p}_{n,n}(x)  {\tilde p}_{n-1,n}(y)   -  {\tilde p}_{n-1,n}(x)  {\tilde p}_{n,n}(y)}{x-y},
\ee
where $\{ {\tilde p}_{k,n} \}$ are the polynomials orthogonal to the weight $w_n(x)$, now restricted to the
complement of $\Gamma$:
\[
    \int_{{\R} \backslash \Gamma}  {\tilde p}_{\ell,n}(x)  {\tilde p}_{k,n}(x) w_n(x) dx = \delta_{\ell k}.
\]
For a large class of potentials  $V$, \cite{DKMVZ99b} employs the Riemann-Hilbert
Problem ($RHP$) characterization of the system $\{ p_{k,n} \}$ to obtain sharp $n \ra \infty$ asymptotics of
the kernel  $K_n$, and thus also the correlation functions $\rho_{n}^{(k)}$, in the bulk
of the spectrum.  Here we take up the analogous project for the Janossy densities at the spectral edge
by analyzing the $\{ {\tilde p}_{k,n} \}$ system.

 The requirements on the potential $V$ are  described in terms
 of the equilibrium measure
 $\mu_V$,  or  weak limit of the eigenvalue counting measure. 
 This may be characterized as 
the infimum of
 \be
\label{energy}
   I _V(\mu) = \int_{-\infty}^{\infty} \int_{-\infty}^{\infty}  \log \frac{1}{| x -  y| } d \mu(x) d \mu(y)
                     + \int_{-\infty}^{\infty} V(x) d \mu (x),
\ee
over  the space of probability measures on $\R$. Now, if
\be
\label{a1}
  V : \R  \, \ra \, \R  {\mbox{ is real analytic}}
\ee
and
\be
\label{a2}
    \lim_{|x| \ra \infty}  \frac{V(x)}{\log(x^2 + 1)}  = + \infty,
\ee
then \cite{DKM98} proves that this infimum is uniquely attained at  $\mu_V$.
Further, $\mu_V$ possesses a density
$\psi_V(x)$ with compact support comprised of a finite number of intervals.
Assumptions (\ref{a1}) and (\ref{a2}) are adopted here.  By a scaling
we may fix the rightmost edge of the support of $\psi_V(x)$ at $x= 1$, and we further
assume that
\be
\label{a3}
        \psi_{V}(x)  \mbox{ is }  { regular}.
\ee
By this we will mean the following.

\medskip

 (a) $\psi_V$ vanishes like a square-root at each endpoint of $\supp(\mu_V)$.

 (b) $\psi_V$ is strictly positive in the interior of  $\supp(\mu_V)$.
   
 (c)  Strict inequality holds in the characterizing Euler-Lagrange equations in the 

\hspace{.45cm} exterior of $\supp(\mu_V)$,  see (\ref{gp2}).
    
\medskip     
 
 \noindent
 One imagines
that square-root vanishing at  $x = 1$ would suffice; full regularity has been assumed for technical reasons.

For $V(x)$ satisfying (\ref{a1}), (\ref{a2}) and (\ref{a3}),  one may infer from the results in  \cite{DKMVZ99b} 
that the  kernel $K_N$  at the spectral edge has the universal limit,
\be
\label{Airy}
  \lim_{n \ra \infty}   \frac{1}{c_V n^{2/3}}   K_{ n}
        \Bigl(1 +  \frac{x}{c_V n^{2/3}}, 1 +   \frac{y}{c_V   n^{2/3} }    \Bigr)   =   \frac{ \Ai(x) \Ai^{\pr}(y)  -  \Ai^{\pr}(x) \Ai(y)}{x-y} ,
\ee
with constant $c_V > 0$ and $\Ai(\cdot)$ the Airy function.\footnote{Though understood to hold in greater generality, a
detailed proof of (\ref{Airy}) actually only appears in the literature for polynomial $V$ \cite{DeiftGioev}.}
     Based on this, it is expected that the
kernel $L_{n, \Gamma}$ for $\Gamma =  [ 1 + \alpha/ ( c_V n^{2/3}), \infty)$ with any real $\alpha$
will have a universal limit as $n \ra \infty$.   We introduce the shorthand,
\be
\label{Ldef}
    L_{n,\alpha}(x, y)  \equiv  L_{n, [1 + \alpha / ( c_V n^{2/3} ),  \, \infty) } (x,y),
\ee
and,
noting that the regime of interest is for $x, y \in \Gamma$,  prove the following.

\begin{theorem}
\label{maintheorem}
Assume that the potential $V(x)$  satisfies (\ref{a1}), (\ref{a2}) and (\ref{a3}).
Then, there are  pairs of functions $\{ {f}_{\alpha}^{\sra}(z), {g}_{\alpha}^{\sra}(z) \}$ and 
$ \{ {f}_{\alpha}^{\sla}(z), {g}_{\alpha}^{\sla}(z) \}   $
defined for $\alpha > 0$ and $\alpha \le 0$ respectively,
such that the following universal asymptotics hold.  For $\alpha > 0$,
\be
\label{AAlim}
 \frac{1}{c_V n^{2/3}} L_{ n, \alpha}
        \Bigl(1 +  \frac{x}{c_V n^{2/3} }, 1 +  \frac{y}{ c_V   n^{2/3} }    \Bigr)
     =  \frac{ {f}_{\alpha}^{\sra}(x)  {g}_{\alpha}^{\sra}(y)
    -  {g}_{\alpha}^{\sra}(x)  {f}_{\alpha}^{\sra}(y)} { x - y} + \OO ( n^{-2/3}  ),
\ee
while for $\alpha \le 0$,
\be
\label{BBlim}
 \frac{1}{c_V n^{2/3}} L_{ n, \alpha}
        \Bigl(1 +  \frac{x}{c_V n^{2/3}} , 1 +   \frac{y}{ c_V   n^{2/3} }    \Bigr)
     =  \frac{     {f}_{\alpha}^{\sla} (x-\alpha)   {g}_{\alpha}^{\sla} (y - \alpha)
    -     {g}_{\alpha}^{\sla}(x -\alpha)   {f}_{\alpha}^{\sla}(y - \alpha)} { x - y} + \OO ( n^{-2/3} ).
\ee
Both estimates are uniform for $x$ and $y$ restricted to compact sets of $(\alpha, \infty)$.
\end{theorem}

The functions ${f}_{\alpha}^{\lra}(z)$
and ${g}_{\alpha}^{\lra}(z)$ are read off from the
solutions of a pair of $ 2 \times 2$  $RHP$s denoted by $RHP^{\sra}$  for $\alpha > 0$ and $RHP^{\sla}$  for $\alpha < 0$.
With the corresponding contours and their  orientations depicted in Figure \ref{firstcontours}, we have:

\vspace{.2cm}

\noindent{\boldmath{$ RHP^{\sra} \;  (\alpha > 0)$}: }  Seek a $ 2 \times 2 $ matrix valued function
$M^{\sra}(z)$, analytic in $\C \backslash \Sigma^{\sra}$ such that:
\be
\ba{ll}
     (M^{\sra})_{+}(z) =  (M^{\sra})_{-}(z) \left( \ba{cc}  1 & e^{ -  \frac{4}{3} z^{3/2} }    \\
                                                0 & 1    \ea \right),  &  z  \in (0, \alpha), \\
     (M^{\sra})_{+}(z) =  (M^{\sra})_{-}(z) \left( \ba{cc}  1 & 0 \\ e^{    \frac{4}{3} z^{3/2} }  &
                                                1    \ea \right),  &   \arg z  = \pm \frac{2}{3} \pi, \\
     (M^{\sra})_{+}(z) =  (M^{\sra})_{+}(z) \left( \ba{rr}  0 & 1 \\
                                                      -1 &   0    \ea
                              \right),
      &  z  \in (-\infty,0)
\ea
\ee
with
\be
\label{Masymp1}
  \hspace{.9cm}   M^{\sra}(z)  = z^{- \frac{1}{4} \sigma_3}  \,
  \frac{1}{\sqrt{2}}    \left( \ba{rr}  1 & 1 \\
  -1 &  1    \ea \right) e^{- \frac{\pi i}{4} \sigma_3}  ( I + \OO ( z^{-1} )  ),
  \  \  \    z \ra \infty.
\ee

\vspace{.2cm}

\noindent{\boldmath{$RHP^{\sla} \; (\alpha  < 0)$}:}  Now seek a $ 2 \times 2 $ matrix valued function
$M^{\sra}(z)$, analytic in $\C \backslash \Sigma^{\sla}$ such that:
\be
\ba{ll}
    (M^{\sla})_{+}(z) =  (M^{\sla})_{-}(z) \left( \ba{cc}  1 & 0 \\ e^{    \frac{4}{3} z^{3/2} + 2 \alpha z^{1/2} }  &
     1  \ea \right),  &   \arg z  = \pm \frac{2}{3} \pi, \\
    (M^{\sla})_{+}(z) =  (M^{\sla})_{+}(z) \left( \ba{rr}  0 & 1 \\
                                                        -1 &   0    \ea
                              \right),  &  z  \in
                              (-\infty,0),
\ea
\ee
with the same asymptotics as $z \ra \infty$.

Note that the problems coincide at $\alpha = 0$.  In either problem,
 $(M^{\lra})_{\pm}(z)$ indicate the limits of $ M^{\lra}(z)$ as $z$ approaches either $\Sigma^{\sra}$ or
$\Sigma^{\sla}$ from the positive or negative sides (the precise sense in which the limit holds is discussed later). Last, $\sigma_3$
denotes the third Pauli matrix,  $\left( \begin{smallmatrix} 1 & \phantom{-} 0 \\  0 & -1 \end{smallmatrix} \right)$.


\begin{figure}[t]
\centerline{   
             \scalebox{0.55}{
             \includegraphics{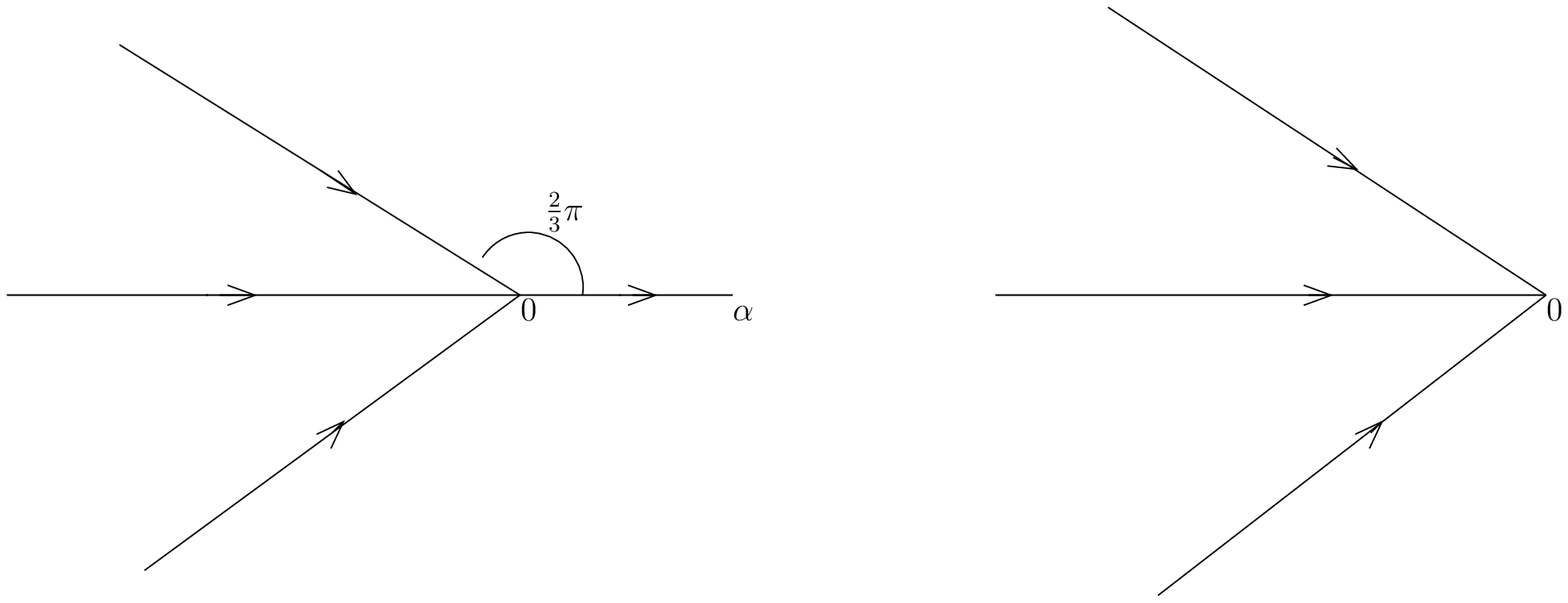}
                }
             }
\caption{The contours  $\Sigma^{\sra}$ and $ \Sigma^{\sla}$ for $RHP^{\sra}$ and $RHP^{\sla}$}
\label{firstcontours}
\end{figure}

A large part of this paper is dedicated to the proof that there exist
unique solutions to $RHP^{\sra}$ and $RHP^{\sla}$.  Granting that, we
may now define the functions comprising the limiting kernels (\ref{AAlim})
and (\ref{BBlim}), hereafter denoted   $\mathbb A_{\alpha}(x,y)$ and $\mathbb B_{\alpha}(x,y)$.

\begin{definition}
\label{maindef}
For $ -\frac{2}{3} \pi  < \arg z < \frac{2}{3} \pi$ and  $z \notin [0, \alpha]$,
\be
\label{AAdef}
     ( { f}_{\alpha}^{\sra} (z), {g}_{\alpha}^{\sra}(z) )  
         =  \frac{1}{\sqrt{2\pi}} e^{ \frac{\pi i}{4}}  e^{ - \frac{2}{3} z^{3/2}}
       \left(  ({M}^{\sra})_{11}(z),   ({M}^{\sra})_{21}(z)  \right).
\ee
Similarly,
\be
\label{BBdef}
      ( {f}_{\alpha}^{\sla}(z), {g}_{\alpha}^{\sla}(z) )  
   =   \frac{1}{\sqrt{2\pi}} e^{ \frac{\pi i}{4} } e^{  -(  \frac{2}{3} z^{3/2} +  \alpha  z^{1/2} ) }
             \left( ({M}^{\sla})_{11}(z) ,   ({M}^{\sla})_{21}(z) \right)
\ee
for all   $z$ with $ - \frac{2}{3} \pi  < \arg z < \frac{2}{3} \pi$.  
\end{definition}

One concludes that $  {f}_{\alpha}^{\lra}(x)$ and $  {g}_{\alpha}^{\lra}(x) $
are real analytic for $x > \alpha$ and $x > 0$; the diagonals
 $\mathbb A_{\alpha}(x,x)$ and $\mathbb B_{\alpha}(x,x)$ for $x > \alpha$
are therefore well defined.  As for their 
 behavior as functions of $\alpha$:

\begin{theorem}
\label{contthm}
Both $\mathbb A_{\alpha}(x,y)$ and $\mathbb B_{\alpha}(x,y)$ are continuous
functions of $\alpha$ for fixed $x$, $y$.
Continuity holds down (or up) to $\alpha = 0$ from either side.
\end{theorem}

Finally, 
while we have not expressed the limit kernel
in terms of known special functions, we do have the following asymptotics.

\begin{theorem}
\label{asympcor}
Uniformly for $z$ in  compact sets of $(0, \infty)$, 
\be
\label{AAasymp}
  { f}_{\alpha}^{\sra}(z)  = {\Ai} (z)  \Bigl( 1 +   \OO(e^{- \alpha^{3/2}  } )   \Bigr),  
   \   \   \ {g}_{\alpha}^{\sra}(z)  = {\Ai} (z)  \Bigl( 1 +   \OO(e^{- \alpha^{3/2}  } )   \Bigr)
\ee
as $\alpha \ra + \infty$, while
\begin{eqnarray}
\label{BBasymp}
      {f}_{\alpha}^{\sla}(z) &  =  &  ( |\alpha| - \frac{2}{3} z )^{1/2}  I_0{( z^{1/2} ( |\alpha| - \frac{2}{3} z ) )}
     \Bigl( 1 +   \OO( |\alpha|^{-1} )   \Bigr), \\
      {g}_{\alpha}^{\sla}(z)   & = &  -2 \pi   I_0^{\pr}{( z^{1/2} ( |\alpha| - \frac{2}{3} z ) )} 
     \Bigl( 1 +   \OO( |\alpha|^{-1} )   \Bigr)  \nonumber
\end{eqnarray}
as $\alpha \ra - \infty$.  $I_0(\cdot)$ is  the  modified  Bessel function of the first kind.
\end{theorem}

After describing applications of Theorem \ref{maintheorem} to
the limiting distributions of the largest eigenvalues for Unitary
ensembles and some possible extensions, the analysis begins  in Section 2 
where the $RHP$ connected to the polynomials $\{ \tilde{ p}_{k,n} \}$ 
is introduced.
Section 3 subjects this  $RHP$ to a series of transformations, following
the Deift-Zhou method of steepest descent \cite{DZ93}.
A local analysis for the problem in the vicinity of $z =1$ in terms of
$RHP^{\sla}$ and $RHP^{\sra}$ is detailed in Section 4. With these
parametrices, Theorem \ref{maintheorem} is proved in Section 5.
Section 6 is devoted to the existence question for $RHP^{\sla}$ and
$RHP^{\sra}$; this is accomplished by a general vanishing lemma argument.
Section 7 establishes several properties of those solutions, 
including their continuity and asymptotics (Theorems
\ref{contthm} and \ref{asympcor}).

\medskip

\noindent
{\bf Remark} The results here should be compared with those in the recent
paper \cite{ClaeysArno} which considers the following set-up:
Take $V$   regular  with right-most edge of $\mu_V$
placed at the origin, and seek the asymptotics of the corresponding orthogonal polynomial kernel 
for the weight $e^{-nV(x)}$ restricted to $(-\infty, 0]$.   This is the same starting point as
our problem.  In \cite{ClaeysArno} though, a parameter $c = c(n) > 0$ is introduced in the weight 
as in
$cV$ which, when adjusted, can move the edge of the support of $\mu_V$ away from the 
origin or push it in, creating a ``hard-edge", or square-root singularity at the origin.  While different 
at all finite $n$, this device
is  qualitatively the same as our choice of  the sign of $\alpha$, and should lead to the same 
phase transition in the limit.  By a quadratic transformation, the authors of \cite{ClaeysArno} are able
to write the limiting  kernel for points $x$ and $y$ to the left of the origin in terms of  Painlev\'e II.
For the probabilistic motivations here, it is the kernel to the right of the critical point
which is important.   By analogy this should correspond to the kernel in \cite{ClaeysArno} along the
imaginary axis; the precise relationship remains to be worked out.

\subsection{Janossy densities and the distribution of the largest eigenvalues}

Denote the ordered eigenvalues of $M$ by $\lambda_1 > \lambda_2 > \cdots$.
The well known gap formula for determinantal ensembles with kernel $K_n$
states that: for any $B \subset \R$,
\be
\label{gap}
  P \Bigl(\mbox{there are exactly } m \mbox{ eigenvalues in } B  \Bigr)
   =    \frac{-1^m}{m!} \frac{d^m}{d \theta^m} \,
       \det  \Bigl( I - \theta    K_n   {\one}_B  \Bigr)  \Bigr|_{\theta = 1}.
\ee
With  $m = 0$, this formula together with the limiting result (\ref{Airy}) impies
\be
\label{largest}
  \lim_{n \ra \infty}  P \Bigl( \lambda_{1} \le 1 +  \frac{\alpha}{c_V n^{2/3}} \Bigr)
     = \det \Bigl(  \ID - K_{Airy}    {\one}_{[\alpha, \infty)} \Bigr),
\ee
with $K_{Airy}$ standing in for the $L^2$-operator with Airy kernel, see 
again \cite{DeiftGioev} for a full proof in the case of polynomial $V$.
The celebrated result of Tracy and Widom (\cite{TW94}, with extensions in
\cite{TW96}) provides a closed form for this Fredholm determinant,
to wit,
\be
\label{TW}
  \det \Bigl(  \ID - K_{Airy}    {\one}_{[\alpha, \infty)} \Bigr) =
    \exp  \Bigl( - \int_{\alpha}^{\infty} (s - \alpha) u^2(s) ds  \Bigr)  \equiv F_{TW}(\alpha),
\ee
in which $u(s)$ is the unique solution of  Painlev${\acute{\mbox{e}}}$ II
with $u(s) \sim \Ai(s) $ as $s \ra +\infty$.

Formulas for
the limiting distributions of the scaled $\lambda_2$, $\lambda_3$, {\it etc},
also exist.    This is also found in \cite{TW94}, though note there the asymptotics 
are only taken on for GUE. To explain, 
first replace the appearance of $u(s)$ in (\ref{TW})
with $u(s; \theta)$ determined by the same equation but with 
$u(s; \theta) \sim \sqrt{\theta} \Ai(s)$ at infinity, and denote the corresponding 
exponential function $F(\alpha; \theta)$.  Then,  following (\ref{gap}),
$ \frac{-1^m}{m!} \times  \partial_{\theta}^{(n)} F(\alpha; \theta) $
evaluated at $\theta = 1$
yields the (limiting) probability of there being exactly $m$ eigenvalues larger
than $\alpha$.   The corresponding distribution functions can  then be built in the
obvious manner.

The Janossy densities provide a different path to the law
of the scaled largest eigenvalues.
From the definition (\ref{Janossy}), we have that
\beqn
P \Bigl(\mbox{exactly } m \mbox{ eigenvalues in } B  \Bigr)
& = & P \Bigl( \mbox{no eigenvalues in  }  B \Bigr) \\
&   &  \times \frac{1}{m!} \int_{B} \cdots \int_{B}
          \det \Bigl[ L_{n, B} (x_{\ell}, x_k)\Bigr]_{1 \le \ell ,k \le m}
          d x_1 \cdots d x_{m},
\eeqn
and,  assuming (\ref{largest}),  one can obtain the following from
Theorem \ref{maintheorem}.

\begin{cor}
\label{undertheint}
With ${\mathbb M}_{\alpha}$ equal to  $ \mathbb A_{\alpha}$ for $\alpha> 0$ and
${\mathbb B}_{\alpha}$ for $\alpha < 0$,
\be
\label{janlim}
\lim_{n\ra \infty}  P \Bigl(  \lambda_m \le 1 +  \frac{\alpha }{ c_V n^{2/3}}   \Bigr)
= F_{TW}(\alpha) \times   \sum_{n= 0}^{m-1} \frac{1}{n!} \int_{\alpha}^{\infty} \cdots \int_{\alpha}^{\infty}
         \det \Bigl(  {\mathbb M}_{\alpha}(x_{\ell}, x_{k} ) \Bigr) dx_1 \cdots d x_{n}.
\ee
\end{cor}

This describes the general limit distribution
as that of the largest eigenvalue 
modulated by a finite sum of (standard) determinants with a universal kernel.
Again, at this point the kernel is only defined in terms of a pair of $RHP$s, and so
the above form of limit law is far from optimal.

\begin{proof}[Proof of Corollary \ref{undertheint}]
We  provide just a sketch, using  estimates developed below.
 To produce the second factor in (\ref{janlim}),
one must pass the point-wise convergence of the kernel established 
Theorem \ref{maintheorem} under the
integral.
The fast decay  of the exponential weight will control the integral at infinity, while more information is
needed to deal with the integral near $\alpha$.  In particular, here one wants to show
that $n^{-2/3} L_{n,\alpha}(1+ n^{-2/3}x, 1+ n^{-2/3} x)$ is uniformly integrable over 
$x \in [\alpha, \alpha + \epsilon]$.  
 Local forms of the solution of $RHP^{\lra}$, provided 
in (\ref{6.8}) for $\alpha > 0$ and (\ref{Qdef1})-(\ref{Qdef3}) for $\alpha < 0$, show that the first columns of $M^{\lra}$
along with their derivatives are bounded down to $\alpha$ or the origin (from the right).   
By Definition \ref{maindef}
 we then see that, even on diagonal, the limit kernel (${\mathbb M}_{\alpha}(x,x)$
 $= {f}_{\alpha}'(x)g_{\alpha}(x)  - {f}_{\alpha}(x) {g}_{\alpha}'(x)$, 
 neglecting the ($\lra$)-superscripts)
 is integrable near $\alpha$. Next,
(\ref{bMdef}) and (\ref{Lker2}) 
 express the finite $n$ kernel 
 in terms of certain (well-behaved) auxiliary functions and the first column of $M^{\lra}(z)$.
A simple analysis of those auxiliary functions shows that $n^{-2/3} L_{n,\alpha}$ inherits the  integrability of  $M^{\lra}$ and
yields the result.
\end{proof}

Improved asymptotics for the Janossy kernel 
might also provide estimates on the speed of convergence
to the Tracy-Widom law (analogues of either the Berry-Esseen estimates or Edgeworth expansions
for the classical central limit theorem).  Results of this type are important in multivariate statistics,
and have already been established for GUE and the related LUE in \cite{Kar04} and \cite{Cp06}.
The case of  unitary ensembles with non-quadratic potentials has not been explored. 
Note however from the general formula (\ref{LK}) we have
\beq
\label{logdet}
  \frac{d}{d \alpha} \log P \Bigl( \lambda_1 \le 1 +  \frac{\alpha}{c_V n^{2/3}}  \Bigr)
   & = & \frac{d}{d \alpha} \log  \det \Bigl(  \ID - K_n \one_{ [1 + \alpha c_V^{-1} n^{-2/3}, \infty)} \Bigr) \\
   & = & \frac{1}{c_V n^{2/3} }  L_{n, \alpha} \Bigl( 1  +  \frac{\alpha}{c_V n^{2/3}}, 1 +  \frac{\alpha}{c_V n^{2/3}} \Bigr).
   \nonumber
\eeq
A similar expression at $n = \infty$ is a first step in the  derivation of (\ref{TW}), and  
the limiting kernels (\ref{AAlim})
and (\ref{BBlim}) are not surprisingly tied  to the resolvent kernel of the Airy operator,
see \cite{D99a} and \cite{TW98}.   More to the point,  a suitable expansion in $n$ in
(\ref{logdet}) would bound the convergence speed.
From Theorem \ref{maintheorem} one anticipates the rate is $n^{-2/3}$,  and that is 
just what is proven for
 GUE and LUE.
  Of course,  
carrying out the suggested program requires  sharp 
asymptotics of $L_{n, \alpha}(x,y)$  along the diagonal ($x = y = \alpha$).
An estimate of the form
$$
   \Bigl|  \frac{1}{c_V n^{2/3} }  L_{n, \alpha} \Bigl( 1  +  \frac{\alpha}{c_V n^{2/3}}, 1 +  \frac{\alpha}{c_V n^{2/3}} \Bigr)
             - {\mathbb M}_{\alpha}(\alpha, \alpha) \Bigr| \le n^{-2/3} \phi(\alpha)
$$
with  $\phi(\alpha)$ integrable at positive infinity would, for example, more than suffice.

\section{First RHP and introduction to the calculation}
\setcounter{num}{2}
\setcounter{equation}{0}
\label{sec:theprob}

The starting point is  the $RHP$ characterization of orthogonal polynomials
due to Fokas, Its and Kitaev \cite{FIK92}.
Fix a half-line  $\Gamma = (-\infty, c]$ and consider the polynomials
\be
 \{  {\tilde p}_{k,n}  =  \tilde \gamma_{k, n}  x^{k} + \cdots \}
 \mbox{ orthonormal with respect to }
  w_n(x) = e^{-n V(x)}  \mbox{ for }  x \in \Gamma.
\ee
Then, the $RHP$ reads as follows.

\vspace{.2cm}

\noindent{\boldmath{$RHP$}  \bf{for} \boldmath{$Y$}:}
Seek a $2 \times 2$ matrix valued function $Y(z) = Y_n(z)$ such that
\be
\label{YRHP}
\ba{ll}
    Y(z)  \mbox{ analytic in }  \C \backslash \Gamma, &  \\
     Y_{+}(z) = Y_{-}(z)  \left( \ba{cc}  1 & w_n(z) \\
                                  0 & 1  \ea \right),    & z \in \Gamma, \\
     Y(z) =  \Bigl( I + O \Bigl( \frac{1}{z} \Bigr) \Bigr) \left( \ba{cc}  z^{n} & 0 \\
                    0 & z^{-n}  \ea \right), &  z \ra \infty.
\ea
\ee

The second, or ``jump",  condition, is read as
\be
\label{boundaryvalues}
    Y_{\pm}(z) \equiv \lim_{z^{\pr} \ra z} Y(\zeta), \ \ \  z \in \Gamma, \ \  \    z^{\pr}  \in \C_{\pm},
{\mbox{ the upper or lower half-plane}}.
\ee
This can be understood in the sense of continuous boundary values for $z \in \Gamma$
away from the endpoint $z = c$  with the additional condition that
\[
   Y(z) = \left(   \ba{cc}  \OO(1)  & \OO( \log|z-c| ) \\  \OO(1) &  \OO(\log|z-c| )  \ea  \right),  \    \   \  z \mbox{ near } c.
\]
With that said,  the basic result  is that the unique solution
of this $RHP$ is given by
\beq
\label{Ysol}
  Y(z) & = & \left( \ba{cc}   \frac{1}{\tilde{\gamma}_{n,n} } {\tilde p}_{n,n}(z) &
    \frac{1}{\tilde{\gamma}_{n,n} } \,  C  \Bigl( {\tilde p}_{n,n} w_n \Bigr)  (z)   \\
                    - 2 \pi i \, {\tilde \gamma}_{n-1,n} \,  {\tilde p}_{n-1,n} (z) &
                     -  2 \pi i  \,{\tilde \gamma}_{n-1,n}   \, C  \Bigl(  {\tilde p}_{n-1,n} w_n \Bigr)  (z)
               \ea \right),
\eeq
where  $C$ denotes the Cauchy operator on $\Gamma$:
\[
         C f (z)  = C_{\Sigma} f (z) \equiv  \frac{1}{2\pi i}  \int_{\Sigma}  \frac{ f(s)}{s -z} \, {ds},  \   \   \ z \notin \Sigma,
\]
for any contour $\Sigma \subset \C$ and function $f(z) \in L^2(\Sigma, |dz|)$.

Note that (\ref{Ysol}) contains the $(w_n, \Gamma)$ orthogonal polynomials of degrees $n-1$ and $n$
in its first column.   It follows that the kernel of interest,  $L_{n, \alpha}$, may be expressed entirely
in terms $(Y_{11}(z), Y_{21}(z) )$ where  $\Gamma$ is now a function of both $n$ and $\alpha$:
we have in particular,
\be
\label{GammaDef}
    \Gamma = \Gamma_{n,\alpha} = \left( - \infty,  1 +  \alpha  c_V^{-1} n^{-2/3} \right],
\ee
and will use the additional shorthand $ c_{n,\alpha} \equiv  \alpha c_V^{-1} n^{-2/3} $ for the
(moving) endpoint.

The analysis of $Y$ for $n \ra \infty$ entails a series of transformations,
$ Y \, \mapsto \, T \, \mapsto \, S \, \mapsto \,  R$,
in order to obtain a $RHP$ for $R$ which is normalized at infinity ({\em i.e.},
$R(z) \ra I$ as $z \ra \infty$),  and has  jump matrices which are uniformly close
to the identity as $n \ra \infty$.  Afterwards, unfolding this series of transformations
will produce the asymptotics of $Y$.  We are primarily concerned with the behavior
of $Y(z)$ in the vicinity of $z=1$, for which we will build local parametrices.
The basic program is identical to that in the analysis of the $RHP$
connected to orthogonal polynomials over the full line in \cite{DKMVZ99b}
or \cite{DKMVZ99a}.  Novel here is that the problem will follow two
different paths, depending on the sign of  $\alpha$.

\subsection{Equilibrium measures and the  $g$-function}
\label{equimeasure}

The first transformation, $Y \mapsto T$, rests on properties of the
the density  which minimizes the analogue of $I_V$ in
(\ref{energy}).  We begin by recalling several properties of $\psi_V$, the {``unconstrained"}
 equilibrium density connected to the analysis of the full-line orthogonal polynomials.

 As indicated in the introduction, the support of
  $\psi_V$ is a union of ($N+1$) disjoint intervals  and we normalize the right endpoint
  to sit at $1$.  The  intervals of support are referred to as the  {\em bands};  the
  complementary $N$ intervals making up the {\em gaps}.  Following \cite{DKMVZ99b},
  the interior of the support is denoted by
  \[
       J =  \bigcup_{k=1}^{N+1} (b_{k-1},  a_k),
  \]
 and the density $\psi_V$ can be written,
 \be
 \label{equdens}
    \psi_V(z) = \frac{1}{2 \pi i}  R_+^{1/2}(z)  h_V(z),  \  \   \  \mbox{ for }  z \in J,
 \ee
 in which
 \be
 \label{Rfactdef}
     R(z) =  \prod_{k=1}^{N+1} (z - b_{k-1}) ( z - a_{k} ),
 \ee
and $h_V$ is real analytic on ${\R}$.   Here, the branch of $R(z)$ is chosen so that
$R(z)$ behaves like $z^{N+1}$ as $z \ra  \infty$.

With this, the first transformation of the $RHP$ for the unconstrained polynomials
is based on the introduction of the {\em g-function},
\be
\label{oldg}
    g(z) = \int_{\R} \log(z-x) \psi_V(x) \, dx.
\ee
The function $g(z)$ is analytic on $\C \backslash (-\infty, 1]$ and has the following properties.
First, there a constant $\ell$ so that
\beq
\label{gp1}
             g_+(z)  + g_{-} (z)  -   V(z)   - \ell & = & 0,  \  \  \   \mbox{ for }  z \in {\bar J},  \\
\label{gp2}
             g_ {+}(z) + g_{-}(z) -  V(z) - \ell   & <   &   0, \  \  \  \mbox{ for } z \in {\R} \backslash {\bar J}.             
\eeq
(The strict inequality in (\ref{gp2}) is our final regularity condition.)
Second, it holds that
\be
\label{gp3}
    g_+(z) - g_{-}(z) = 2 \pi i \int_{z}^1 \psi_V(x) \,dx, \ \ \ \mbox{ for } z \in (-\infty, 1).
\ee
That is, $g_+(z) - g_{-}(z)$ is purely imaginary on $\R$ and constant in each of the gaps,
with the more detailed picture being:
\be
\label{gp4}
   g_{+}(z)  - g_{-}(z) =  \left\{     \ba{ll} 2 \pi i, & z  \in (-\infty, b_0),  \\
                                                                  2 \pi i \int_{b_k}^1 \psi_V(x) dx \equiv 2 \pi i \Omega_k,
                                                                            &  z \in (a_k, b_k), \\
                                                                       0,    &  z \in (1, +\infty).  \ea \right.
\ee
Appraisals
(\ref{gp1}) through (\ref{gp4}) are basic consequences of the Euler-Lagrange
equations for (\ref{energy}), as is explained in Section $3.2$ of \cite{DKMVZ99b}.

With  $\{ {\tilde p}_{k,n} \}$ we are working on the $n$-dependent interval
$\Gamma_{n,\alpha}$,   but  may  proceed in a like manner.  For each integer $n$ and real $\alpha$, the
old reasoning will show that infimum of
\be
\label{energycut}
   I_{n,\alpha}(\mu) =  \int_{\Gamma_{n, \alpha}}  \int_{\Gamma_{n,\alpha}} \log \frac{1}{|x - y|}  d \mu (x) d \mu (y)
             +  \int_{\Gamma_{n, \alpha}  } V(x) d \mu  (x)
\ee
is uniquely achieved.   The minimizing density is however qualitatively different
$\alpha > 0$ or $\alpha < 0$, and is denoted by $\psi_V^{\sra}$ or $\psi_V^{\sla}$
respectively.

\subsubsection{The case $\alpha > 0$}

When $\alpha > 0$, $\Gamma_{n,\alpha}$ contains the support of the full-line minimizer $\psi_V$,
and we state without proof the following.

\begin{lemma}
\label{firstdensitylemma}
It holds that $\psi_V^{\sra}(z) = \psi_V(z) $. Thus,  by assumption (\ref{a3}),
\be
\label{densatend1}
   \psi_V^{\sra}(x) =  ( 1 - x)^{1/2}  \beta_V(x) ,   \   \    x  \in  (1-\delta, 1),
\ee
for all small $\delta > 0$.  Here, $\beta_V(z) $ is analytic in a neighborhood of $z =1$,
$\beta_V(1)  > 0$, and for later we remark $ c_V \equiv (  \beta_V(1) /2)^{2/3}$.
\end{lemma}

Note that a definition of $\beta_V(z)$ is implicit in (\ref{equdens}).  Also, (\ref{densatend1})
has been set apart as this regime is of central importance in what follows.

\subsubsection{The case $\alpha < 0$}

For $\alpha< 0$, one attains different minima in (\ref{energycut}) and (\ref{energy}).
Still, using the assumed regularity we will show that,
for all $n$ large enough,  
the support of $\psi_V^{\sla}$ consists of
$N+1$   intervals,
\[
    J_{n, \alpha} =     \bigcup_{k = 1}^{N+1} \Bigl( b_{k-1}(n,\alpha), a_{k}(n, \alpha) \Bigr),
\]
with $a_{N+1}$ fixed at $c_{n,\alpha} = 1 + \alpha/(c_V n^{2/3})$.  Now setting,
\be
\label{Rndef}
   \wt R_{n, \alpha}(z)  = \frac{(z- b_N)}{ (z - a_{N+1})} \,  \prod_{k=1}^{N} (z - b_{k-1}) ( z - a_k),
\ee
(from here on we suppress the $(n, \alpha)$-dependence of the endpoints)
we have:

\begin{lemma}
\label{densitylemma2}
For $\alpha < 0$ and $V$ satisfying (\ref{a3}) it holds,
\be
\label{density2}
  \psi_{V}^{\sla} (z)  = \frac{1}{2 \pi i }    (\wt R_{n, \alpha} (z))_+^{1/2}  \,
              \Bigl[  (c_{n,\alpha}   -  z )  h_{V, n, \alpha}(z)  + C_{n,\alpha} \Bigr] ,
\ee
on its support,  with
\[
    C_{n, \alpha} =   \frac{1}{2}  ( \frac{ \alpha }{ c_V n^{2/3}} )    \, h_V(1)  + \OO(n^{-4/3}),
\]
where $h_V$ is as in (\ref{equdens})
and $h_{V, n, \alpha}(z)$ is real analytic and tends to $h_V$ as $n\ra \infty$.  The analogue of (\ref{densatend1}) reads
\be
\label{densatend2}
\psi_V^{\sla}(x) =    (c_{n, \alpha} - x)^{1/2}   \, \beta_{V, n, \alpha}^{(a)}(x)  +   \frac{1}{2} ( \frac{\alpha}{c_V n^{2/3}}) 
 ( c_{n, \alpha} - x)^{-1/2}  \, \beta_{V, n, \alpha}^{(b)}(x), \  \ \
  x \in [1 - \delta, c_{n,\alpha}],
\ee
where $ \beta_{V, n, \alpha}^{(a,b)}(z)$ are analytic in a neighborhood of
$z=1$ and $\lim_{n \ra \infty} \beta_{V, n, \alpha}^{(a,b)} (c_{n,\alpha}) = \beta_V(1).$
\end{lemma}

With $\psi_V^{\sla}$ in hand, we define a $g$-function exactly as in (\ref{oldg}).  The basic relations
(\ref{gp1}) through (\ref{gp2}) remain valid, though with adjusted values for $\ell$ and the $\Omega$'s.

\begin{proof}[{Example  (GUE)}]
It is instructive to first spell out the computation for $GUE$, or when $V(z) = 2 z^2$.
The unrestricted minimizer has one band of support, $[-1,1]$, and is given
by the semi-circle law, $\psi_{GUE}(x) = \frac{2}{\pi} \sqrt{ 1 - x^2}$.
Now fixing the right edge at $a = 1- \ep$ for small $\ep > 0$,
there is still one band $[b, a]$ with $b$ and the adjusted density $\psi_{GUE}^{\ep}$ to be
identified.  If we put
\[
   G(z)  = \frac{1}{\pi i } \int_{\R}  \frac{\psi_{GUE}^{\ep}(s)}{s - z} \,  ds ,
\]
differentiating
the relation (\ref{gp1})  produces the scalar $RHP$:
\be
\label{babyrhp}
   G_{+}(z) + G_{-}(z) =   \frac{4 i}{\pi}  z  , \  \   z \in [b,a],
\ \mbox{ and } \  G_{+}(z) - G_{-}(z) = 0, \  \  z \in \R \backslash   [b,a].
\ee
Introduce
$\wt R(z) = \wt R_{\ep}(z)  \equiv \frac{z-b}{z-a}$ and  multiply
both  sides of (\ref{babyrhp}) through by the square-root
of this object.  The $RHP$  is then transformed into a standard
form,  and one finds
\be
   G(z) =
             \frac{ \sqrt{\wt R(z)}}{2 \pi i}    \int_b^a   \frac{ (4 i s / \pi )}{ (\sqrt{\wt R(s)})_{+}}
           \frac{ds}{ s - z},
\ee
subject to the single moment condition,
\be
\label{moment}
   \frac{\pi}{2 i} =  \int_b^a s \sqrt{ \frac{a-s}{s-b}}  ds,
\ee
which holds since $z G(z) \ra  \frac{1}{\pi i }$ as $ z \ra \infty$.
The integral (\ref{moment}) is easily computed, and
$b = b(a) = \frac{1}{3}(a - 2 \sqrt{a^2 + 3})$  for positive $a \le 1$.
Also, by properties of the Stieltjes transform, it holds that
\[
   \psi_{GUE}^{\ep} (x) = \RE ( G_{+}(x) )   =   \frac{2}{\pi} \sqrt{ \frac{x -b }{a - x} }
  \Bigl( \frac{a -b }{2} - x \Bigr) ,  \   \  x \in (a, b). \\
\]
Now, for $\ep$ small  and  $x$ near $a$,
\be
     \psi_{GUE}^{\ep}(x)   =    \left(    ( 1- \ep - x)^{1/2}
              +  \frac{\ep}{2}  (1 - \ep - x)^{-1/2}  \right)  \,({2 \sqrt{2}}/{\pi}) \, (1+  \OO(\ep) ) ,
\ee
after substituting $a = 1- \ep$, $b = - 1 +  \OO(\ep^2)$ in the previous display.
This provides a model for the general formula (\ref{densatend2}).
\end{proof}

\begin{proof}[Proof of Lemma \ref{densitylemma2}]  As in the previous example, set $\ep = \alpha c_V^{-1} n^{-2/3}$ with
$a_{N+1} = a_{N+1}(\ep) = 1 - \ep$.    Following the standard approach,
the density is given by
\[
   \psi_{V,\ep}^{\sla}(z) =   \RE \Bigl( (G_{\ep})_+(z)  \Bigr),
\]
where,
\be
\label{Gep}
   G_{\ep}(z) =  \frac{1}{2 \pi i} \,   
   \frac{\sqrt{ R_{\ep}(z) }}{(z - a_{N+1}(\ep))}  \int_{J_{\ep}}   \frac{ i  V^{\pr}(s) / \pi }{(\sqrt{ R_{\ep}(s)})_+  }  ( a_{N+1}(\ep) - s)    \frac{ds}{s - z},
\ee
and
\[
   R_{\ep}(z) =  \prod_{\ell= 1}^{N+1} (z - b_{k-1}(\ep) ) ( z - a_{k}(\ep) ),  \   \    \   J_{\ep} =  \bigcup_{k=1}^{N+1} ( b_{k-1}(\ep) , a_{k}(\ep) ).
\]
We have normalized in this manner as, when $\ep \ra 0$, the endpoints $(b_k(\ep), a_{k}(\ep) )$  converge to their positions in the full line equilibrium
density $\psi_V$, and $R_{\ep}(z)$ and $J_{\ep}$ converge to $ R(z)$ and $J$ defined in (\ref{Rfactdef}) and directly above.

The integral in (\ref{Gep}) over $J_{\ep}$ can be replaced by
\[
   G_{\ep}^0(z)   =  \frac{1}{2}  \int_{\cal C}  \frac{ i  V^{\pr}(s) / \pi }{\sqrt{ R_{\ep}(s)}  }  ( a_{N+1}(\ep) - s)    \frac{ds}{s - z}
\]
for $\mathcal C$ a clockwise oriented contour surrounding both $J_{\ep}$ and $z$.  Now,
\beq
 \int_{{\cal C}}   \frac{ i  V^{\pr}(s) / 2 \pi }{\sqrt{ R_{\ep} (s)}  }  ( a_{N+1}(\ep) - s)    \frac{ds}{s - z}
 & = &
  ( a_{N+1}(\ep) - z )   \int_{{\cal C} }   \frac{ i  V^{\pr}(s) / 2 \pi }{\sqrt{ R_{\ep} (s)}  }   
  \frac{ds}{s - z}  -  \int _{\cal C}   \frac{ i  V^{\pr}(s) / 2\pi }{\sqrt{ R_{\ep} (s)}  }   {ds}
   \nonumber \\
 & \equiv &  ( a_{N+1}(\ep) - z ) \, h_{V,\ep}(z) +    C_{\ep},
\eeq
where $h_{V,\ep}(z)$ and $C_{\ep}$ are the same as $h_{V, n, \alpha}(z)$ and $C_{n,\alpha}$ in the statement of the Lemma. Since $h_V(z)$ figuring in
the definition of $\psi_V$ is exactly $\lim_{\ep \ra 0}  h_{V,\ep}$ (recall (\ref{equdens})),  it is left to prove that
\be
\label{Cep}
   C_{\ep} =  \int_{\cal C}   \frac{   V^{\pr}(s) }{\sqrt{ R(s)}  }  \frac{ds}{2 \pi i}    - \frac{\ep}{2} \,  \int _{\cal C}   \frac{ V^{\pr}(s)  }{\sqrt{ R(s)}  }  \frac{1}{s-1}  \frac{ds}{2 \pi i}
   + \OO(\ep^2).
\ee
Indeed,
\be
  C_0 =   \int_{\cal C}   \frac{  V^{\pr}(s)}{\sqrt{ R(s)}  }  \frac{ds}{2 \pi i}   \equiv 0,
\ee
by a moment condition for the full-line (or ``free") problem,  and the second integral in (\ref{Cep}) is exactly $h_V(1)$.
Further, the  
 endpoints $\{ b_k(\ep) \}$ and $\{ a_k(\ep) \}$  turn out to be  real analytic functions of $\ep$, 
 and we have 
\[
  C_{\ep} =   \frac{\ep}{2} \int_{\mathcal C}  \frac{V^{\pr}(s)}{ \sqrt{ R(s)} } \Bigr(  \sum_{k=1}^{N+1}     \frac{a_k^{\pr}(0)}{s - a_k(0)} +
  \frac{ b_{k-1}^{\pr}(0)}{s - b_{k-1}(0)} \Bigr)  \frac{ds}{2 \pi i}
  + \OO(\ep^2).
\]
The advertised $\OO(\ep)$ term will then arise from $a_{N+1}(0) = 1$ and $a_{N+1}^{\pr}(0) = -1$, and additional fact 
that  all other endpoints have vanishing first derivative at $\ep = 0$.

The idea behind verifying these last claims  is to view our constrained problem  as a perturbation of
the free problem at the modulation point $P$, where $a_{N+1}=1$.
Returning to (\ref{Gep}) set
$   \tilde R(z)=R(z)/(z-a_{N+1})^2$,
suppressing the dependence on $\ep \ge 0$.   
 For the constrained problem, the system of $2N+1$ modulation  equations 
 determining the endpoints are
 \be
\label{G1}
   T_j \equiv  \int_{J}\frac{V^{\pr}(s)}{\sqrt{ \tilde
   R(s)}_+  } s^j{ds}=0,\ \ \ 0\le j\le N-1,
\ee \be \label{G2}
    T_N \equiv  \int_{J}\frac{V^{\pr}(s)}{\sqrt{ \tilde
     R(s)}_+ } s^N{ds}= \frac{1}{\pi i},
\ee
the so-called moment conditions  along with the integral conditions
\be \label{G3}
   N_k \equiv \int_{\gamma_k}G(z) dz =0, \ \ \ 1\le k\le N,
  \ee where $\gamma_k$ denotes the loop around
    the cut $(b_{k-1},a_{k})$.
 Adding the additional relation
      \be
\label{G4}
   T_{-1} \equiv  \int_{J}\frac{V^{\pr}(s)}{\sqrt{ \tilde
   R(s)} _+}  \frac{ds}{s - a_{N+1}}=0
  \ee
gives the system for the free problem.  In particular, the moment conditions for the full problem
read $ \int_J V'(s)/\sqrt{R(s)}_{+}  s^j ds = 0$ for $0 \le j \le N$  and  $(\pi i )^{-1}$ for $j = N+1$, and
$T_{-1}$ takes you from the system (\ref{G1})--(\ref{G2}) to this one.

  Our regularity assumption for the free problem implies that the
  absolute value of the  $(2N+2)\times (2N+2)$ Jacobian of the
  map $\{a_j, b_k \} \mapsto \{ T_j, N_k\}$
  is bounded below by a positive constant at point $P$.  A full proof may be found in Section 12 of \cite{KM00}.
  In order to conclude all other endpoints are real-analytic functions of $a_{N+1}$, and so $\ep$, we must show
   that the $(2N+1)\times (2N+1)$ Jacobian of
  the modulation equations for the constrained problem is also
   bounded below by a positive constant at point $P$. (And then invoke the implicit function theorem). Clearly, this
   will hold if
   $$
       \frac{\partial T_j}{\partial
   a_{N+1}}   =  0  \mbox{ for } 0 \le j \le N,  \mbox{ and }   \frac{\partial N_k }{\partial
   a_{N+1}} =  0 \mbox{ for }  1 \le k \le N,
   $$
   as, if so, the $(2N+1) \times (2N+1)$ minor must be non-degenerate, less the full Jacobian is. 
   Note also that by a simple application of the chain rule, this will imply
    the vanishing of the first derivatives of all endpoints other than $a_{N+1}$ at
   $\ep = 0$, a fact we used above.  
   
   Finally,  
    $\frac{\partial }{\partial
   a_{N+1}} T_j  = \int_J V'(s)/\sqrt{R(s)}_{+}  s^j ds $ and the latter vanishes for $j \le N$
   by the moment conditions for the free problem.   Also, by linearity the vanishing 
   of the derivative of any $N_k$  will follow from
   $\frac{\partial G}{ \partial
   a_{N+1}}(P)=0$. We compute
   \beqn
   \frac{\partial G}{\partial
   a_{N+1}}  & = &  \frac{-1}{4 \pi i}   \sqrt{\tilde R(z) }\int_{J}\frac{V^{\pr}(s)}{\sqrt{
   R(s)} _+}  \frac{ds}{s - z}+\frac{1}{4 \pi i}
   \frac{\sqrt{\tilde R(z) }}{(z-a_{N+1})}\int_{J}\frac{V^{\pr}(s)(s-a_{N+1})}{\sqrt{
   R(s)} _+}  \frac{ds}{s - z}  \\ \nonumber
  & = &   \frac{1}{4 \pi i}\frac{\sqrt{\tilde R(z)
}}{(z-a_{N+1})}\int_{J}\frac{V^{\pr}(s)}{\sqrt{
   R(s)} _+}  {ds}, \nonumber
   \eeqn
   and note
   the last integral vanishes at point $P$ as the integral appearing reduces to that in condition (\ref{G4}).
   \end{proof}

\section{Steepest descent}
\setcounter{num}{3}
\setcounter{equation}{0}
\label{transformations}

\subsection{First transformation $ Y \mapsto T$}

As in \cite{DKMVZ99b} {Section 3.3},  we define
$T(z)$  for either $\alpha > 0$ or $\alpha<0$ by conjugation,
\be
\label{Tdef}
    T(z) =  e^{- \frac{n \ell}{2} \sigma_3 }  \, Y(z) \, e^{ \frac{ n \ell}{2}  \sigma_3}
          e^{-   {n}  g(z) \sigma_3 } ,
\ee
where $g(z)$ is the log-transform of either $\psi_V^{\sra}$ or $\psi_V^{\sra}$, and we
recall that $ \sigma_3 = \Bigl(  \ba{rr} 1 & 0 \\ 0 & -1 \ea \Bigr)$.    The jump
matrix for $Y(z)$ is transformed into,
\[
   V_T(z) =   \left(  \ba{cc}  e^{-n(g_+(z)  - g_{-}(z) )}  &  e^{n( g_+(z)  + g_{-}(z) - V(z) - \ell)} \\
                                                        0                    &    e^{-n(g_+(z)  - g_{-}(z) )}   \ea \right) ,   \   \   \ z \in \Gamma_{n, \alpha}.
\]
Next, using the relations (\ref{gp1}) and (\ref{gp4}) satisfied by
$g(z)$  and the fact that $e^{ng(z)}  \approx z^n$ for $z \ra \infty$,
we find that $T(z)$ is the unique solution of the following $RHP$.

\vspace{.2cm}

\noindent{\boldmath{$RHP$}  \bf{for} \boldmath{$T$}:}
We seek $T(z)$   analytic in  $ \C \backslash \Gamma_{n, \alpha}$,
with jump relations,
\be
\label{Tjumps}
\ba{ll}
T_+(z)  =  T_{-}(z)  \Bigl(   \ba{cc} e^{- n (  g_{+}(z)  -  g_{-}(z) )} & 1  \\ 0 &
                                           e^{n (  g_{+}(z)  -  g_{-}(z)  ) } \ea  \Bigr),
                    &                  z \in  \bar{J}   \\
T_+(z)  =  T_{-}(z)  \Bigl(   \ba{cc}  e^{-2 \pi i n \Omega_j }
    &  e^{- n (  g_{+}(z)   +  g_{-}(z)   - V(z) - \ell )  }  \\ 0
                                          &  e^{2 \pi i n \Omega_j}  \ea  \Bigr),
                                      &
                                             z \in (a_j, b_j), j = 1, \dots, N,
                                            \\
   T_+(z)  =  T_{-}(z)  \Bigl(   \ba{cc}  1
    &  e^{- n (  g_{+}(z)   +  g_{-}(z)   - V(z) - \ell )  }  \\ 0
                                          &  1  \ea  \Bigr),
                      &   z < b_0  \mbox{ or }   a_{N+1} < z < c_{n,\alpha},
\ea
\ee
and asymptotics,
\be
\label{Tasymp}
T(z) =  I +  \OO \Bigl( \frac{1}{z} \Bigr)   \   \   z \ra \infty.
\ee
That is, we now have an $RHP$ normalized at $\infty$.

\subsection{Second tranformation $ T \mapsto S$}
\label{Sproblemsec}

This step is the descent, transforming the oscillatory diagonal entries of the jump matrices
in the $RHP$ for $T(z)$ into exponentially decaying off-diagonal entries in an equivalent
problem for a function $S(z)$.

For the $\alpha > 0$ case we follow \cite{DKMVZ99b} without change.
For   $z \in  \C \backslash \Gamma_{n, \alpha}$ in the region of analyticity of $h_V(z)$, recall
(\ref{equdens}), define
\[
    \phi(z) =   \int_{a_{N+1}}^z  {R^{1/2}(s)} h_V(s) \, ds,
\]
where $a_{N+1} = 1$ and the path of integration does not cross $\Gamma_{n,\alpha}$.
From (\ref{gp1})  and (\ref{gp4}) we have that, for each $z \in(b_{j-1}, a_j) \subset J$,
\beq
  g_{+}(z) - g_{-}(z)  & =  & 2 \pi i \int_z^{a_{N+1}} \psi_V(s) ds    \\
                                   & = & \int_z^{ a_{j}} R_{+}^{1/2}(s) h(s) ds +  2 \pi i \int_{b_j}^{a_{N+1}}  \psi(s) \, ds \nonumber
                                     \\
                                   & = & -   \phi_{+}(z) =    \phi_{-}(z).  \nonumber
\eeq
That  is, $ -  \phi(z) $ and $  \phi(z)$ are analytic continuations of $g_{+}(z)  - g_{-}(z)$
above and below each band.
Also, $\phi_{+}$ and $\phi_{-}$ are purely imaginary on each band and an easy exercise
using the Cauchy-Riemann conditions shows that,
\be
\label{rephi}
\RE \phi(z) < 0,    \mbox{  for small  }  \IM (z)   \neq 0 \mbox{ and  }  \RE( z) \in J.
\ee
Since the $g$-function tied to $\psi_V^{\sla}$ (for $\alpha < 0$) satisfies the same basic relations
we can define extensions for $
 g_{+}(z) - g_{-}(z)   =   2 \pi i \int_z^{c_{n,\alpha}} \psi_V^{\sla}(s) ds$, and so $\phi$,  in the same
 way.

With these properties of $\phi$ in mind, the jump contour is deformed off the line by opening
a lens around each band based on the factorization
\beqn
  \left(   \ba{cc}     e^{-n(g_{+}(z) - g_{-}(z))} & 1 \\ 0 &  e^{n(g_{+}(z) - g_{-}(z))}  \ea  \right)
 & =  & \left(   \ba{cc}     1 & 0 \\ e^{ n \phi_{-}(z) }  & 1 \ea  \right)
     \left(   \ba{cc}     0 & 1 \\ -1 &  0  \ea  \right)
     \left(   \ba{cc}       1 & 0 \\ e^{ n \phi_{+}(z)}  & 1 \ea  \right) \\
     & \equiv &  { B}_{-}^{-1}(z)  \left(   \ba{cc}     0 & 1 \\ -1 &  0  \ea  \right)  { B}_{+}(z).
\eeqn
Set,
\be
\label{Sdef}
    S(z) =  \left\{   \ba{ll} T(z),  &  z \mbox{ in the exterior of each lens}, \\
                                        T(z)    B_{\pm}^{-1}(z),
                                                  &  z \mbox{ in the upper/lower part of each lens}, \ea \right.
\ee
in which what is meant by a lens,
 and the resulting contour with lenses $\Sigma_{n, \alpha}$, is spelled out in
Figure \ref{openfig}.
We are led to:

\vspace{.2cm}
\noindent{\boldmath{$RHP$}  \bf{for} \boldmath{$S$:}}
$  S(z)$   is   analytic in $  \C \backslash \Sigma_{n, \alpha}$,
satisfies  $
     S(z) = I + \OO(\frac{1}{z} )$  as  $ z \ra \infty$ with   $z \notin \Sigma_{n, \alpha}$,
along with the jump relations,
 \be
\ba{ll}
\label{Sjumps1}
    S_+(z) = S_{-}(z)  \left(   \ba{cc}       1 & 0 \\ e^{n \phi(z)}  & 1 \ea  \right),   & z \in \Sigma_{n, \alpha} \cap \C_{\pm}, \\
    S_+(z) = S_{-}(z)   \left(   \ba{cc}      0 & 1 \\  - 1 & 0  \ea  \right),   & z \in  J, \\
    S_+(z)  = S_{-}(z)  \Bigl(   \ba{cc}  e^{-2 \pi  i \Omega_j}
    &  e^{- n ( g_{+}(z)   +  g_{-}(z)   - V(z) - \ell )  }  \\ 0
                                          &  e^{ 2 \pi i \Omega_j}   \ea  \Bigr),   &  z \in (a_j, b_j) , \, j = 1, \dots, N,  \\
    S_{+}(z)   = S_{-}(z)  \Bigl(   \ba{cc}  1
    &  e^{- n ( g_{+}(z)   +  g_{-}(z)   - V(z) - \ell )  }  \\ 0
                                          &  1  \ea  \Bigr),   &
                                       z < b_0  \mbox{ or }
                         a_{N+1} < z <   c_{n,\alpha}.
\ea
\ee
Note that (\ref{rephi}) implies the factor $e^{ n \phi(z)}$  in (\ref{Sjumps1}) decays exponentially as $n \ra \infty$.
Further, in the regular case $V$  the exponent $ g_{+}(z)  + g_{-}(z) - V(z) - \ell$ appearing in the third and fourth
jump matrix  is strictly negative in the gaps or past the ends of support.
It follows that  the corresponding entries also decay exponentially  as $n \ra \infty$.


\begin{figure}
 \centerline{   
             \scalebox{0.65}{
             \includegraphics{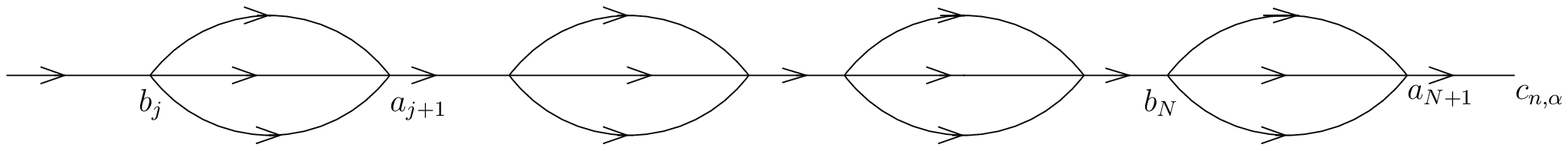}
               }
            }
 \caption{The new contour $\Sigma_{n, \alpha}$ (pictured for $\alpha >0$) with a lens opened around each band.}
 \label{openfig}
 \end{figure}

\subsection{Model Problem}
\label{model}

From the discussion at the end of the previous section we expect the
leading order asymptotics to be governed by the $2 \times 2$ matrix
$P^{\infty}(z)$ which solves the following model problem.

\vspace{.2cm}

\noindent{\boldmath{$RHP$} \bf{for} \boldmath{$P^{\infty}$}:}
$ P^{\infty}(z)$ is analytic in  $\C \backslash  [b_0, a_{N+1} ]$,
$ P^{\infty}(z) = I + \OO(1/z)$,  as $   z \ra \infty,$ and
\be
\ba{ll}
\label{RHPPinf}
  P_{+}^{\infty}(z)  =   P_{-}^{\infty}(z)  \left( \ba{rr} 0 & 1 \\  - 1 & 0 \ea \right),   &  z  \in  J,  \\
  P_{+}^{\infty}(z)  =   P_{-}^{\infty}(z)  \left( \ba{rr} e^{ 2 \pi i \Omega_j}
    & 0 \\  0 & e^{-2 \pi i \Omega_j} \ea \right) ,   &    z   \in ( a_j, b_{j} ),  \   j = 1, \dots, N. \\
\ea
\ee
Though we are led to this problem by considering  the $n \ra \infty$
for the  jumps of $S(z)$, the bands and gaps over which the jumps
of $P^{\infty}$ are defined should
still be taken in their finite $n$ positions for the $\alpha < 0$ case.
\vspace{.2cm}

While the particulars of $P^{\infty}$ will not affect the parametrix we eventually construct
about $z=1$, it is required to demonstrate that the above  problem does indeed
have a solution.  Fortunately, this has already been accomplished
in \cite{DKMVZ99b}, where  it is proved that (\ref{RHPPinf}) has a unique
solution 
satisfying $\det P^{\infty}(z) \equiv 1$.

\subsection{Last transformation $S \mapsto R$}

Since the convergence of the jumps for $S(z)$ to those for $P^{\infty}$ is not uniform near the endpoints, we have to perform
a local analysis at each of the endpoints $a_j, b_j$.   The analysis
at $a_{N+1}$ is particular to the present endeavor and is the subject of the next section,  for the rest though we may
again refer to \cite{DKMVZ99b}.

For $n$ large enough, each  $x_0 =  a_j, b_j$ is regular and there will be no interior ``singular" points.
Surround each $x_0$  
 by a small disk and consider the set of local RHP's: 
 \be
\ba{l}
   P_{x_0}(z)  \mbox{ analytic in }  \{ |z - x_0| < \ep^{\pr} \} \backslash \Sigma_{n,\alpha}  \mbox{ for a } \ep^{\pr} > \ep, \\
   P_{x_0}(z) \mbox{ and } S(z) \mbox{ share jump conditions on } \Sigma_{n,\alpha} \cap \{ |z-x_0| < \ep \}, \\
   P_{x_0}(z) (P^{\infty})^{-1}(z) = I + \OO(n^{-\kappa}), \mbox{ uniformly for } |z- x_0| = \ep \mbox{ with a } \kappa > 0.
\ea
\ee
The last condition matches asymptotics of $P_{x_0}(z)$ to those outside its disk.
Granted a solution we define
\be
\label{Rdef}
  R(z) = \left\{ \ba{ll} S(z) (P^{\infty})^{-1}(z),  &  z  \mbox{ outside the disks,} \\
                                   S(z) (P_{\bullet})^{-1}(z)  &  z \mbox{ inside the disks,} \ea \right.
\ee
in which $P_{\bullet}(z)$ stands in for whichever $P_{x_0}(z)$ corresponds to the given disk.
At all $x_0 \neq a_{N+1}$,  it is well known
$P_{\bullet}(z)$ is given explicitly in terms
of Airy functions;  the connected $RHP$ is in fact $RHP^{\sra}$ with  $\alpha = \infty$,
and the error exponent is $\kappa = 1$.  Assuming that parametrices for 
 $z = a_{N+1}$ exist, $R(z)$ will be analytic
off the system of contours described by Figure \ref{Rcontours}.  (While there is an isolated singularity
in the second column of $R(z)$ at $a_{N+1} = c_{n, \alpha}$ traced back to that in $Y(z)$ at the same point, it is 
logarithmic and so removable).  Next,
since $S(z)$ and $P^{\infty}(z)$ are normalized at infinity and $\det P^{\infty}(z) = 1$, it follows from
its definition that $R(z)$ is also normalized at infinity.   These facts are key ingredients of the result
that, uniformly for $z \in \C \backslash \Sigma_R$,
\be
   R(z)  = I + \OO(n^{-\kappa}),     \    \   \  n \ra \infty,
\ee
with proof identical to that in \cite{DKMVZ99b}.
It follows that $R(z)$,  and also $\frac{d}{dz} R(z)$, are uniformly bounded for large $n$, and
from (\ref{Rdef}) that $\det (R(z)) \equiv 1$.


\begin{figure}[h]
\centerline{   
             \scalebox{0.55}{
             \includegraphics{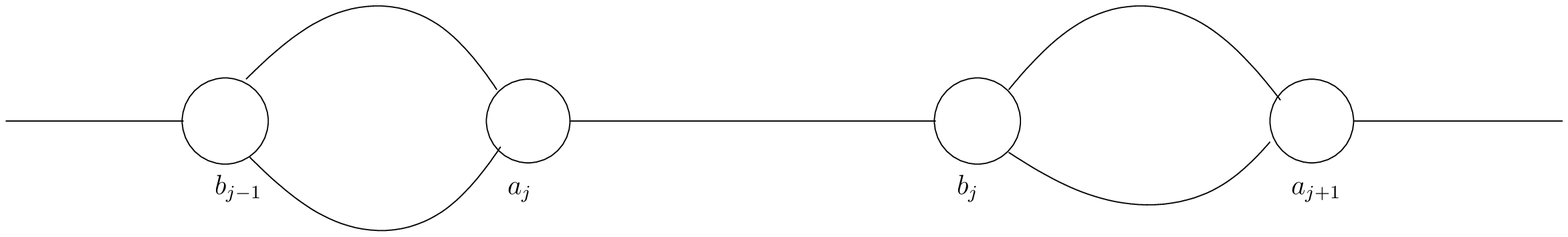}
                }
             }
\caption{Part of the contour $\Sigma_R$.}
\label{Rcontours}
\end{figure}

\section{Paramatrices at the right edge}
\setcounter{num}{4}
\setcounter{equation}{0}
\label{paramatrix}

At last we build  paramatrices for the RHP's
($\alpha > 0$ or $\alpha < 0$)  in a neighborhood of $z = a_{N+1}$
as $n \ra \infty$; these 
are described in terms of the solutions
 of $RHP^{\sra}$
and $RHP^{\sla}$.

When $\alpha > 0$, we have that  $a_{N+1} \equiv  1 < c_{n, \alpha}$.  It is convenient at this point
to bring in the  fact that,
\be
\label{phitoright}
  g_+(z) + g_{-}(z) + V(z) - \ell =  -  \int_{1}^z  {R^{1/2} (s)} h_V(s)  ds  = -  \phi(z),
\ee
for real $z$ with $|z| > 1$.  Then, for a fixed
$U_{\ep} = \{ z : |z - 1| < \ep \}$
we require a $2 \times 2$  $P^{\sra}(z)$  which is analytic in $U_{\ep} \backslash
\Sigma_{n,\alpha}$ and satisfies,
\be
\label{inpar}
 \ba{ll} ( P^{\sra}(z))_{+} = ( P^{\sra}(z))_{-}  \left( \ba{cc}  1  &  0 \\ e^{n \phi(z)}& 1  \ea  \right),
     &  z \in ( \Sigma_{n, \alpha}  \cap U_{\ep})  \cap \C_{\pm},      \\
   ( P^{\sra}(z))_{+} = ( P^{\sra}(z))_{-}  \left( \ba{cc}    0 & 1 \\ -1  & 0 \ea  \right),  &
    z \in (1 - \ep, 1),    \\
   ( P^{\sra}(z))_{+} = ( P^{\sra}(z))_{-}  \left( \ba{cc} 1 & e^{-  n \phi (z)} \\
                                                                 0 & 1     \ea  \right),   &    z \in  [ 1 , 1 +  \frac{{\alpha}}{c_V {n^{2/3}}} ),
\ea
\ee
with
\be
\label{inparasymp}
\hspace{-1.0cm}
    P^{\sra}(z) (P^{\infty}(z))^{-1}  =  I + \OO \Bigl( \frac{1}{n} \Bigr),  \hspace{1.7cm}
         z \in  \partial U_{\ep} \backslash \Sigma_{n,\alpha}.
\ee
The asymptotics (\ref{inparasymp})
entail a matching condition between the inner and outer solutions.

 If instead $\alpha < 0$, the endpoint depends on $n$ as in
 $ a_{N+1} = 1 + \alpha / c_V n^{2/3} = c_{n,\alpha}$. It is
 convenient though
 to keep the same neighborhood $U_{\ep}$ fixed about $z = 1$ and the problem is to find
 $P^{\la}(z)$, analytic in $U_{\ep} \backslash \Sigma_{n, \alpha}$ with jump conditions
\be
\label{outpar}
 \ba{ll} ( P^{\sla}(z))_{+} = ( P^{\sla}(z))_{-}  \left( \ba{cc}  1  &  0 \\ e^{ n \phi(z)}& 1  \ea  \right),
     &  z \in  ( \Sigma_{n, \alpha} \cap  U_{\ep} ) \cap \C_{\pm},      \\
   ( P^{\sla}(z))_{+} = ( P^{\sla}(z))_{-}  \left( \ba{cc}    0 & 1 \\ -1  & 0 \ea  \right),  &
    z \in (1 - \ep,  1 + \frac{\alpha}{c_V  n^{2/3}} ),
\ea
\ee
and  again,
\be
\label{outparasymp}
\hspace{-1.3cm}
    P^{\sla}(z) (P^{\infty}(z))^{-1}  =  I + \OO \Bigl( \frac{1}{n} \Bigr),  \hspace{1.4cm} z \in  \partial U_{\ep} \backslash \Sigma_{n, \alpha} .
\ee

These problems are now mapped onto those for $M^{{\sra}}(\zeta)$ and $M^{\sla}(\zeta)$.
For $z \in U_{\ep}$ define two changes of variables $z \ra \zeta = \zeta_n^{\lra}(z)$, via
\be
\label{varchange1}
      n \phi(z)  = \frac{4}{3} \zeta^{3/2},  \   \   \   \mbox{ for }  \alpha > 0,
\ee
and
\be
\label{varchange2}
      n \phi (z)  =  \frac{4}{3} \zeta^{3/2}  + 2 \alpha  \zeta^{1/2},   \  \   \   \mbox{ for }   \alpha < 0.
\ee
Also, set
\be
\label{matching}
   E_n(z)  =  P^{\infty}(z)  \frac{1}{\sqrt{2i}}  \left(  \ba{cc}  i & - i \\ 1 & 1 \ea \right)   \,  \left(   \zeta_n(z)  \right)^{  \sigma_3/4}.
\ee

\begin{lemma}
\label{paralemma}
The $RHP$s   ((\ref{inpar}), (\ref{inparasymp}))  and  ((\ref{outpar}),   (\ref{outparasymp}))
are solved by
\be
    P^{\sra}(z)  =  E_{n}(z)  M_n^{\sra} ( \zeta_n^{\sra}(z) )    \   \    \mbox{   and   }   \  \   P^{\sla}(z)  =  E_{n}(z)  M^{\sla} ( \zeta_n^{\sla}(z) ).
\ee
Here $\zeta_n^{\sra}(z)$ and $\zeta_n^{\sla}(z)$ is given by (\ref{varchange1}) or (\ref{varchange2}) respectively.
Also, $M_{n}^{\sra}(\zeta)$  represents the
solution of $RHP^{\sra}$ in which the jump over $[0,\alpha)$ is replaced by the same jump over $[0, \alpha_n)$
with  $\alpha_n \, \ra \, \alpha$ as $n \ra \infty$  defined below in (\ref{movingboundary}).
\end{lemma}

\begin{proof} For $\alpha > 0$, from (\ref{densatend1}) we have that: with $z$ close to $1$,
\be
\label{posphi0}
   n \phi(z) =   n  \int_{1}^z (s-1)^{1/2}  \beta_V(s) ds  =   \frac{4}{3} n (s- 1)^{3/2}  \, \wt{\beta}(z),
\ee
where $\wt{\beta}(z)$ inherits the analticity properties of $\beta_V$.  Now set
\be
\label{posphi}
  \zeta_n(z) =    n^{2/3}  (z - 1)  (  \wt{\beta}(z)   )^{2/3}.
\ee
Along with being analytic in  $U_{\ep}$,  $\wt{\beta}(1) > 0$, and the branch may be chosen so that,
$(\wt \beta(1))^{2/3} > 0$.  Then, by choice of $\ep$,  $\zeta_n(z)$ maps  $U_{\ep}$
 one-to-one and onto an open neighborhood of  $\zeta = \zeta_n(1) = 0$.  Further,  $\zeta_n(U_{\ep} \cap \R)
 \subset \R$,  $\zeta_n(U_{\ep} \cap \C_{\pm}) \subset \C_{\pm}$, and those parts of the
$z$-contour $(\Sigma_{n,a}  \cap U_{\ep}) \cap \C_{\pm}$ can be chosen so that their images
are $\arg \zeta \equiv  \pm \frac{2}{3} \pi$.  Last, $z \in [1, 1 + \alpha/ c_V n^{2/3} ]$ is mapped to
\be
\label{movingboundary}
   \zeta \in \left[  0,   \alpha \, (  \wt{\beta}( c_{n,\alpha}) )^{2/3} / {c_V})  \right] \equiv
  [ 0, \alpha_n]   \,  \ra  \, [ 0 ,     \alpha  (   \beta_V(1) /2  )^{2/3} / c_V ]   = [0, \alpha],
\ee
as $n \ra \infty$, providing the definition of  $\alpha_n$.

When $\alpha < 0$,   we have from (\ref{densatend2})  that,
\beq
\label{negphi}
   n \phi(z) & = &    n   \int_{c_{n,\alpha}}^z   \Bigl( ( s-  c_{n, \alpha})^{1/2}  \beta_{V, n, \alpha}^{(a)} (s)   - \alpha c_V^{-1} n^{-2/3}
                                (s - c_{n,\alpha})^{-1/2}   \beta_{V,n, \alpha}^{(b)}(s)  \Bigr) ds     \\
                   & = &   \Bigl\{ \frac{4}{3}  n ( z-  c_{n, \alpha})^{3/2}     +  \left( \frac{2 \alpha}{ c_V}   n^{1/3} + \OO( n^{-1/3} )  \right)
                                (z - c_{n,\alpha})^{1/2}  \Bigr\}  \wh \beta(z).
                                \nonumber
\eeq
Note the change of sign:  $(s - c_{n, \alpha})_+^{1/2} =    -   (s - c_{n,\alpha})_+^{-1/2}$.
Here  again $\wh \beta(z) $  is analytic and positive in a fixed neighborhood of $1$, and so $c_{n,\alpha}$,
for all large enough $n$.   (That $\beta_{V,n,\alpha}^{(a,b)}(z)$ are analytic near $1$ and differ by $\OO(n^{-2/3})$
at $z = c_{n,\alpha}$ is used.)  Choosing  the ${2/3}$-root of $\wh \beta$ positive in that same neighborhood, it follows that
(\ref{varchange2}) has the required properties: $U_{\ep}$   is mapped one-to-one and onto a neighborhood
of $\zeta= \zeta_n(c_{n,\alpha})  = 0$,  $\zeta_n( \R \cap U_{\ep})$ is real, and the segments of $\Sigma_{n,a}$
can be chosen to map onto $ \arg \zeta = \pm \frac{2}{3} \pi$.

Plainly, $M_{n}^{\sra}(\zeta_n(z))$ and $ M^{\sla}(\zeta_n(z))$ satisfy the jumps specified for $P^{\sra}(z)$ and $P^{\sla}(z)$.
Next, as fully explained in \cite{DKMVZ99a},  $E_{n}(z)$ is analytic in $U_{\ep}$ and so does not affect the jump
relations.  Briefly, $P^{\infty}(z)$ has a singularity of the form $(z - a_{N+1} )^{ - \sigma_3 /4}$ at the right edge,
and this is compensated by the appearance of $( \zeta_n(z))^{\sigma_3/4}$  in (\ref{matching}) and the fact
that $\zeta_n(z) \approx C  n^{2/3} (z - a_{N+1}) $ whether $\alpha$ is positive or negative.  The asymptotics
(\ref{inparasymp}) and (\ref{outparasymp}) follow from the fact that there is a constant $c > 0$
with $|\zeta_n(z)| > c n^{2/3}$ uniformly on $|z| = \ep$ and $n \ra \infty$ and the behavior as $\eta \ra \infty$
of $M^{\lra}(\zeta)$  stated in (\ref{Masymp1}).
\end{proof}

\begin{remark}  Along with $E_n(z)$ being analytic in a neighborhood of $z=1$, it also follows from
the form of $P^{\infty}(z)$ that both $E_n(z)$ and $ \frac{d}{dz} E_n(z) $ are uniformly bounded in
that neighborhood.  Further,  $\det E_n(z) \equiv 1$ since $\det P^{\infty}(z) \equiv 1$.  
\end{remark}

\section{Proof of Theorem \ref{maintheorem}}
\setcounter{num}{5}
\setcounter{equation}{0}
\label{mainlimit}

The derivation below borrows heavily from  \cite{KV03}.  First note the expression for 
 $L_{n, \alpha}$ in terms
of the $Y$ matrix (defined in (\ref{YRHP})):
\be
\label{kernelre1}
    L_{n,\alpha}(z,w)  = - \frac{1}{2 \pi i}   e^{ - \frac{1}{2} V(z) - \frac{1}{2} V(w)} \,
      \frac{Y_{11}(z) Y_{21}(w) - Y_{21}(z) Y_{11}(w)}{ z- w}.
\ee
We are interested in the asymptotics  of the above kernel for real $z$ and $w$
to the right of $a_{N+1}$, the endpoint of support of the density of states.
 Unravelling the sequence
of transformations leading from $Y$ to $R$, we have by (\ref{Tdef}), (\ref{Sdef}) and (\ref{Rdef}),
\be
\label{writeY}
  Y(z) =   Y_{n, \alpha}(z) =   \sqrt{2 \pi}  e^{\frac{\pi i}{4}}
  e^{\frac{n \ell}{2} \sigma_3} R(z)
  E_n(z) {\bf M}_{n,\alpha} ( z )  e^{ \frac{1}{2} n \phi(z)  \sigma_3}
  e^{- \frac{n \ell}{2} \sigma_3}  e^{n g(z) \sigma_3}
\ee
in which we have made the definition:
\be
\label{bMdef}
  {\bf M}_{n, \alpha}(z)  = \left\{  \ba{ll} \frac{e^{\pi i /4}}{\sqrt{2 \pi}} M_n^{\sra}( \zeta_n^{\sra}( z) )
  e^{-\frac{1}{2} n  \phi(z) \sigma_3} &  \mbox{ for }   \alpha > 0 \\
  \frac{e^{\pi i /4}}{\sqrt{2\pi}}    M^{\sla}(\zeta_n^{\sla}(z) )
  e^{-\frac{1}{2} n \phi(z) \sigma_3 } &
  \mbox{ for }   \alpha <  0 \ea \right..
\ee
Here
$\zeta_n^{\sra} (z)$ and $\zeta_n^{\sla}$  are given by (\ref{varchange1}) and (\ref{varchange2}) ,
and $M_n^{\sra}$ is as in Lemma \ref{paralemma}.
It follows that,\footnote{If  $z$, $w$  are taken to the left of $a_{N+1}$, there is an additional factor of
$\left( \ba{cc} 1 & 0 \\ e^{ - n \phi(z)} & 1 \ea   \right)$ in (\ref{writeY}) arising from opening
the lenses in the $T \mapsto S$ step.  In that case,  $Y_{11}$
and $Y_{21}$  are linear combinations of the first and second rows of ${\bf M}_{n, \alpha}$
respectively.}
\be
\label{Ycol}
  \left(  \ba{c}  Y_{11}(z) \\ Y_{21}(z) \ea \right)   =   \sqrt{2 \pi}  e^{\frac{\pi i}{4}}
    e^{n( g(z) -  \frac{\ell}{2} + \frac{1}{2} \phi(z)) }  e^{\frac{n \ell}{2} \sigma_3}
   R(z)  E_n(z)
   \left(  \ba{c}  ({\bf M}_{n,\alpha} )_{11} (\zeta_n(z) ) \\  ({\bf M}_{n, \alpha} )_{21}  (\zeta_n(z)) \ea \right).
\ee
Now,
for either the $(\sra)$ or $(\sla)$ case and  real $z > $  $a_{N+1} = 1$ or $a_{N+1} =  1 - \alpha/ c_V n^{2/3}$,
we have that,
\[
     g(z) -  \frac{1}{2} V(z)  - \frac{\ell}{2}   =  \frac{1}{2} ( g_{+}(z) + g_{-}(z) - V(z) - \ell ) =  - \frac{1}{2} \phi( z),
\]
recall (\ref{phitoright}).
Further, set
\[
     K(z) \equiv  R(z)  E_n(z),
\]
and recall that  both $R(z)$ and $E_n(z)$ have determinant $ = 1$, are analytic in a neighborhood
of $z=1$, and, along with their derivatives, are uniformly bounded there.  Obviously, $K(z)$ inherits these
properties.
Next change variables as in
\[
   z \mapsto x_n =  1 + \frac{x}{c_V n^{2/3}}   ,  \   \    w \mapsto y_n =  1 + \frac{y}{c_V n^{2/3}},
\]
with $x, y $ $ > \alpha$, and  summarizing the steps thus far we have:
\beq
\label{Lker2}
\lefteqn{     \wh L_{n, \alpha}(x_n , y_n)  \equiv    \frac{1}{c_V n^{2/3}} L_{n, \alpha} ( x_n, y_n)} \\
       & =  &  -
      \frac{1}{ 2 \pi i ( x - y) }  \det \left(  \ba{cc}  e^{ - \frac{n \ell}{2} }  e^{- \frac{1}{2} n V(x_n)}   Y_{11}(x_n) &
                                                                               e^{ - \frac{n \ell}{2} }  e^{- \frac{1}{2} n V(y_n)}   Y_{11}(y_n) \\
                                                                                e^{  \frac{n \ell}{2} }  e^{-  \frac{1}{2} n V(x_n)}   Y_{21}(x_n) &
                                                                                e^{  \frac{n \ell}{2} }  e^{- \frac{1}{2} n V(y_n)}   Y_{21}(y_n)   \ea \right)
                                                                                          \nonumber \\
       &  = &  \frac{1}{(x-y)}
                        \det \left[
                K(x_n) \left(  \ba{ll}    ({\bf M}_{n ,\alpha})_{11} ( {x}_n ) & 0
                                                 \\    ({\bf M}_{n,\alpha})_{21}( x_n)  & 0 \ea \right)
              +  K(y_n) \left( \ba{ll}  0 &   ({\bf M}_{n,\alpha})_{11} ({y}_n)       \\ 0 &
                                    ({\bf M}_{n,\alpha})_{21} ({y}_n)                                        \ea \right)
               \right]. \nonumber
\eeq
Before further manipulations we record the following two facts.

\begin{claim}
\label{claim1}
As $n \ra \infty$,
\be
\label{zetalimit}
    \zeta_n^{\ra} ( x_n) =   x ( 1 + \OO( n^{-2/3})),  \ \ \mbox{ and }  \ \       \zeta_n^{\sla}(x_n)  =  (x-\alpha)  ( 1 + \OO(n^{-2/3}) ),
\ee
uniformly for $x$ in compact sets of $(\alpha, \infty)$. Also,
\be
\label{philimit}
   n \phi(x_n)  =   n \phi_{n,\alpha} (x_n) = \left\{   \ba{ll}  \frac{4}{3} x^{3/2} \, ( 1 + \OO( n^{-2/3} ) ), &  \alpha > 0, \\
                                       ( \frac{4}{3}(x-\alpha)^{3/2} + 2 \alpha (x-\alpha)^{1/2} )  (1+ \OO( n^{-2/3}))  & \alpha < 0, \ea  \right.
\ee
with the same uniformity in $x$ as $n \ra\infty$.
\end{claim}

\begin{claim}
\label{claim2}
 For all real $x > \alpha $,
 \be
\label{Mlimit}
   M_{n}^{\sra}(x) -  M^{\sra}(x)  = \OO (n^{-2/3} ),
\ee
and
\be
\label{Mderivlimit}
  \frac{d}{dx} ( M_{n}^{\sra}(x) -  M^{\sra}(x)  ) = \OO (n^{-2/3} ),
\ee
as $n \ra \infty$. The estimates are uniform for $x$ a positive
distance from $\alpha$.
\end{claim}

\begin{proof}[Proofs]   The estimate (\ref{philimit}) follows directly from substituting the definition of $x_n$ into
(\ref{posphi0}) and (\ref{negphi}):
\[
    n \phi(x_n)  =  \frac{4}{3} x^{3/2} \times   \wt  \beta(x_n ) / (c_V)^{3/2},
\]
for $\alpha > 0$, and,
\[
    n  \phi(x_n) =  \Bigl\{ \frac{4}{3} (x - \alpha)^{3/2}  + (2 \alpha + \OO(n^{-4/3}) ) ( x - \alpha)^{1/2}   \Bigr\} \times  \wh \beta(x_n) / (c_V)^{3/2},
\]
for $\alpha < 0$.  Since $\wt \beta(z)$ and $\wh \beta(z)$  are analytic and $= c_V^{3/2} > 0$ at $z =1$, each of the rightmost
factor above may be expanded as in $1 + \OO(n^{-2/3}) + \cdots$.   The same considerations lead to  (\ref{zetalimit}).

As for (\ref{Mlimit}) and (\ref{Mderivlimit}), Lemma \ref{contlemma} below proves that $M^{\sra}(z) = M^{\sra}(z; \alpha)$
is continuous in $\alpha$ for $z$ supported away from $\Sigma^{\sra}$.  The estimate (\ref{contestimate}) obtained in its
proof provides
\[
    M^{\sra}( z ; \alpha) - M^{\sra}(z ; \beta) = \OO( \alpha - \beta),
\]
with the same holding for the derivative of the left hand side. Now, since $ M_n^{\sra} (z) \equiv M^{\sra}( z ; \alpha_n)$
and (\ref{movingboundary}) shows that $\alpha_n = \alpha + \OO(n^{-2/3})$, the claim is proven.
\end{proof}

Picking up the calculation, the matrix within the last determinant of
(\ref{Lker2}) is now written as
\beq
\label{det22}
\lefteqn{   \hspace{-2cm} K(x_n)  \left[ \left(  \ba{cc}  ({\bf M}_{n,\alpha})_{11} (x_n) &  ({\bf M}_{n,\alpha})_{11} ({y}_n) \\
                              ({\bf M}_{n,\alpha})_{21} (x_n) &  ({\bf M}_{n,\alpha})_{21} ({y}_n)          \ea \right) \right.} \\
          & & \ \ \ \
              \left. + K(y_n)^{-1}  ( K (x_n) - K (y_n) )
                           \left(   \ba{cc}    ({\bf M}_{n,\alpha})_{11} (x_n)  & 0  \\
                                         ({\bf M}_{n,\alpha})_{21} (x_n)  & 0  \ea \right) \right].  \nonumber
\eeq
Since the analytic matrix function  $K(z)$ satisfies  $\det K(z) \equiv 1$, we see that the form of desired limit
resides in the first term of (\ref{det22}).  As for the second term, first note that
since $K(z)$  and its derivative are  uniformly bounded for  $|z- 1| <  {\ep}$ with a small
enough $\ep> 0$, it follows that
 $K(z)^{-1}$  is bounded in kind and that
 $K(x_n) - K(y_n) = \OO ( |x_n-y_n| ) = \OO( |x-y| n^{-2/3})$, by the mean-value theorem.
Therefore,
\[
 K(y_n)^{-1}  ( K (x_n) - K (y_n) )  \left(   \ba{cc}    ({\bf M}_{n,\alpha})_{11} (x_n)  & 0  \\
                                         ({\bf M}_{n,\alpha})_{21} (x_n)  & 0  \ea \right) =
                           \left(   \ba{cc}
                                            \OO ( |x-y| n^{-2/3}   )  & 0  \\
                                              \OO ( |x-y| n^{-2/3}   )  &  0
                                              \ea   \right).
\]
Along with the estimates on the pre-factor,
we are in the domain of analyticity of $M_n^{\sra}$ and $M^{\sla}$, which coupled with
(\ref{zetalimit}) through (\ref{Mlimit}), implies that ${\bf M}_{n, \alpha}(x_n)$
is uniformly bounded.

The kernel now reads,
\beq
\label{det33}
  {\wh L}_{n,\alpha}(x_n ,y_n) & = & \frac{1}{(x-y)} \det   \left(
                             \ba{cc}  ({\bf M}_{n,\alpha})_{11} ({x}_n)    + \OO ( |x-y| n^{-2/3}   )  & ({\bf M}_{n,\alpha})_{11} ({y}_n)  \\
                                            ({\bf M}_{n,\alpha})_{21} ({x}_n)    + \OO ( |x-y| n^{-2/3}   )  & ({\bf M}_{n,\alpha})_{21} ({y}_n)
                              \ea \right)   \\
                  & = &   \frac{1}{(x-y)} \det   \left(
                             \ba{cc}  ({\bf M}_{n,\alpha})_{11} ({x}_n)     & ({\bf M}_{n,\alpha})_{11} ({y}_n)  \\
                                            ({\bf M}_{n,\alpha})_{21} ({x}_n)    & ({\bf M}_{n,\alpha})_{21} ({y}_n)
                                    \ea \right)    + \OO (  n^{-2/3}  ),   \nonumber
\eeq
for $x, y$ bounded in $(\alpha, \infty)$ and all large $n$.  Instead of expanding the first term of the right hand side
entry-wise,  the second column of that matrix is subtracted by the first,
\be
\label{det44}
 {\wh L}_{n,\alpha}(x_n ,y_n)  =  \frac{1}{(x-y)} \det   \left(
                             \ba{cc}  ({\bf M}_{n,\alpha})_{11} ({x}_n)   -  ({\bf M}_{n,\alpha})_{11} ({y}_n)  & ({\bf M}_{n,\alpha})_{11} ({y}_n)  \\
                                            ({\bf M}_{n,\alpha})_{21} ({x}_n)   -  ({\bf M}_{n,\alpha})_{21} ({y}_n)  & ({\bf M}_{n,\alpha})_{21} ({y}_n)
                                    \ea \right)    + \OO (  n^{-2/3}  ),
\ee
which will allow an estimate uniform in $x$ and $y$ even as $|x-y|  \downarrow 0$.

Using (\ref{philimit})  and (\ref{zetalimit}) and the analyticity of $M^{\sla}(z)$ ($\RE  z > 0)$, one can check
that: with $\xi =  \frac{e^{ - \pi i/4 }}{\sqrt{2 \pi}}$,
\[
   \frac{d}{dx}  \Bigl( (M_{n, \alpha})_{\cdot 1}(x_n) -
          (M^{\sla})_{\cdot 1}(x-\alpha) \, \xi  e^{ \frac{1}{2} \theta(x-\alpha)}   \Bigr) = \OO  \hspace{-.08cm} \left( \frac{x}{n^{2/3}} \right)
\]
with $\theta(x- \alpha) = (4/3) (x-\alpha)^{3/2} +  2 \alpha (x-\alpha)^{1/2}$ and $\cdot = 1, 2$.
For $\alpha > 0$, the same basic reasoning gives:
\[
    \frac{d}{dx}  \Bigl( (M_{n, \alpha})_{\cdot 1}(x_n) -
          (M_n^{\sra})_{\cdot 1} (x) \, \xi e^{ \frac{1}{2} \theta(x)}   \Bigr) = \OO   \hspace{-.08cm} \left( \frac{x}{n^{2/3}} \right)
\]
with $\theta(x) = (4/3) x^{3/2}$ and again $\cdot = 1$ or $2$.  It follows that
\beq
\label{difbound}
(M_{n, \alpha})_{\cdot 1}(x_n) - (M_{n, \alpha})_{\cdot 1}(y_n) &  =  &  \xi
 \,    \Bigl(  (M^{\sla})_{\cdot 1}(x-\alpha)  e^{ \frac{1}{2} \theta(x-\alpha)}  -  (M^{\sla})_{\cdot 1}(y-\alpha)  e^{ \frac{1}{2} \theta(y-\alpha)}  \Bigr)  \nonumber \\
& & + \OO ( |x - y| n^{-2/3}).
\eeq
for $\alpha < 0$, with the analogous statement for $\alpha > 0$.  Detailing how  (\ref{difbound})
is employed in (\ref{det44}) we recall the definition of $(f_{\alpha}^{\sla} , {g}_{\alpha}^{\sla})$
from (\ref{BBdef}) and write,
\beq
\label{det55}
\lefteqn{  {\wh L}_{n,\alpha < 0}(x_n ,y_n) } \nonumber  \\
& =  &  \frac{1}{x-y}
\det  \left(  \ba{ll}    ( {f}_{\alpha}^{\sla}(x-\alpha)    -  {f}_{\alpha}^{\sla}(y-\alpha)   + \OO(|x-y| n^{-2/3}) &
                             {f}_{\alpha}^{\sla}( y-\alpha)  + \OO(n^{-2/3}) \\
                            {g}_{\alpha}^{\sla}(x-\alpha)  -  {g}_{\alpha}^{\sla}(y-\alpha)  + \OO  (|x-y| n^{-2/3} ) &
                               {g}_{\alpha}^{\sla}(y-\alpha)  + \OO(n^{-2/3})
                          \ea    \right)  \nonumber  \\
& = &  {\mathbb B}_{\alpha}(x,y)
 + \frac{1}{(x-y)} \det \left(  \ba{ll}    {f}_{\alpha}^{\sla}(x-\alpha)    -  {f}_{\alpha}^{\sla}(y-\alpha)   &
                                \OO(n^{-2/3}) \\
                            {g}_{\alpha}^{\sla}(x-\alpha)  -  {g}_{\alpha}^{\sla}(y-\alpha)  &
                               \OO(n^{-2/3}) \ea \right)  + \OO ( n^{-2/3} ).
\eeq
The first term of the right hand side is the advertised limit kernel.  To show that the second term is
$\OO(n^{-2/3})$ uniformly in $x,y$ in bounded sets of $(\alpha, \infty)$ note that
\[
    \frac{{f}_{\alpha}^{\sla}(x-\alpha)  -  {f}_{\alpha}^{\sla}(y-\alpha)}{x-y} = \OO(1),
\]
for all such $x$ and $y$
since $x  \mapsto (M^{\sla})_{11}(x-\alpha) e^{ - \frac{1}{2} \phi(x-\alpha) } $ is smooth for $x > \alpha$.
The same is true if ${f}_{\alpha}^{\sla}$ is replaced by ${g}_{\alpha}^{\sla}$.
This completes the proof for $\alpha < 0$.

In the case
$\alpha > 0$ following the steps behind (\ref{det55}) produces
\[
  {\wh L}_{n,\alpha > 0}(x_n ,y_n)  =  \frac{\xi^2}{(x-y)}  e^{ - (\frac{2}{3} x^{3/2} + \frac{2}{3} y^{3/2} )}
     \det  \left(  \ba{ll}  (M_{n}^{\sra})_{11}(x) &  (M_{n}^{\sra})_{11}(y)  \\
                                    (M_{n}^{\sra})_{21}(x) &  (M_{n}^{\sra})_{21}(y)  \ea \right) + \OO(n^{2/3}).
\]
Next repeat the procedure:   subtracting the second column from the first in the above
determinant and now
employing the estimates of
Claim \ref{claim2} will  allow each appearance  of $(M_n^{\sra})_{\cdot 1}$ to be replaced
with the corresponding  $(M^{\sra})_{\cdot 1}$ with an overall $\OO(n^{-2/3})$ error.

\section{Existence for the local problems  $RHP^{\lra}$}
\setcounter{num}{6}
\setcounter{equation}{0}
\label{existence}

The general theory connects the the construction of a solution
to a given RHP  to that of a certain
singular integral operator.  In particular, consider the RHP $(\Sigma, v)$:
\beq
\label{testRHP}
\ba{ll}
  m(\zeta) \mbox{ analytic in } \C \backslash \Sigma,  &  \\
  m_{+}(\zeta) = m_{-}(\zeta) v(\zeta),  & \zeta \in \Sigma,  \\
  m(\zeta) = I  + \OO \Bigl( \frac{1}{\zeta} \Bigr),   & \zeta \ra \infty, \zeta \notin \Sigma,
\ea
\eeq
in which $v$ is continuous on $\Sigma$ away from points of self-intersection and
$v(\zeta) \ra I$  as $\zeta \ra \infty$ along $\Sigma$.  Next define the integral
operator
\[
  C_{v} f (\zeta) =  C_{-}  \Bigl(  f \, ( v - I  )  \Bigr)
\]
on $L^2(\Sigma, |d \zeta||)$, where $C_{-}$, and $C_{+}$ are the $\pm$-limits of the Cauchy operator:
\[
     ( C_{\pm}f )  (\zeta) =   \lim_{\zeta^{\pr} \ra \zeta \atop \zeta^{\pr} \in \pm \; \mbox{side of }
     \Sigma}  \frac{1}{2 \pi i}  \int_{\Sigma} \frac{f(s)}{ s - \zeta^{\pr}} ds,
       \  \  \ \zeta \in \Sigma.
\]
With  $v \in I + L^2(\Sigma)$,  $C_v$ is bounded from $L^2(\Sigma) \ra L^2(\Sigma)$, and
an $L^2$ solution to (\ref{testRHP})  can be construction out of a solution $\mu_v \in I + L^2(\Sigma)$
of
\be
\label{mudef}
    ( \ID - C_v) \mu_v = I
\ee
via
\be
\label{solrep}
    m(\zeta) = I +   \frac{1}{2\pi i}  \int_{\Sigma}  \frac{ \mu_v(s)  (v(s) - I) }{ s - \zeta}  ds .
\ee
For
details behind these facts,
\cite{Zhou89}  is recommended.  Existence for the $RHP$   $(\Sigma, v)$ would then follow from
showing that $\ID - C_v$ is a bijection in $L^2(\Sigma)$.

To apply this strategy to either
$RHP^{\sra}$ or  $RHP^{\sla} $  requires a preliminary step: it is not the case that
$M^{\lra}$, or their corresponding jump matrices  $V^{\lra}$,  are
normalized to the identity at infinity.
Therefore, we bring in
\be
\label{twistsol}
{\mathfrak m} (\zeta) = {\zeta}^{\frac{-\sigma_3}{4}}  \frac{1}{\sqrt{2}} \,
             \left( \ba{cc} 1 & 1 \\ -1 & 1 \ea \right)  e^{ - \frac{i \pi}{4}  \sigma_3},
\ee
the fundamental solution of the twist problem:
\[
    {\mathfrak m}_+(\zeta) =
    {\mathfrak m}_{-}(\zeta)     \left( \ba{cc} 0 & 1 \\ -1 & 0 \ea \right),  \   \    \  \mbox{ for } \zeta \in \R_{-},
\]
and ${\mathfrak m}(\zeta)$ analytic in $\C \backslash \R_{-}$.
With this, we define
\be
\label{tame}
     \widetilde{M}^{\lra}(\zeta) =    \left\{
        \ba{ll}
                M^{\lra}(\zeta)  \,   {\mathfrak m }^{-1}(\zeta), &     \mbox{ for }  |\zeta| > R,  \\
                    M^{\lra}(\zeta),   &   \mbox{ for }  |\zeta| < R,  \ea \right.
\ee
with a fixed large and positive $R$.   The contours for the pair of RHPs for ${\wt M}^{\lra}$
appear in Figure \ref{infty_out}.  Obviously, in each case the jump along the negative real
axis has been removed far out.  Further, along ${\wt \Sigma}_2$ and ${\wt \Sigma}_4$
and $|\zeta| > R$, the new jumps
\be
\label{tamejumps}
   {\wt V}(\zeta) =  \mathfrak m (\zeta)   \left( \ba{cc} 1 & 0 \\ e^{\tht(\zeta)} & 1 \ea \right) ({\mathfrak m}(z))^{-1},
\ee
with
$ \tht(\zeta) =  \frac{4}{3} \zeta^{2/3}$ or   $ \tht(\zeta) =   \frac{4}{3} \zeta^{2/3} - 2  |\al|  \zeta^{1/2}$,
still
decay exponentially fast to the identity as $\zeta \ra \infty$.  Last, along the introduced
contour $|\zeta| = R$, both
problems have the uniformly bounded jump ${\mathfrak m}^{-1}(\zeta)$.


\begin{figure}[h]
\centerline{   
             \scalebox{0.58}{
             \includegraphics{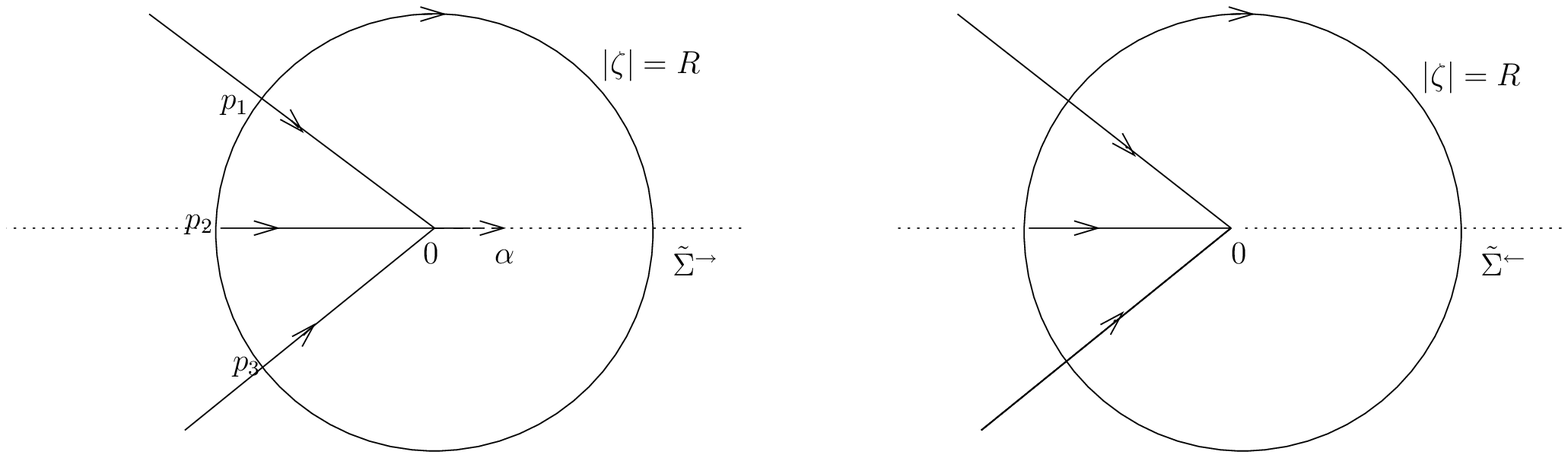}
                }
             }
\caption{The contours for $( \wt \Sigma^{\lra}, \wt V^{\lra} )$.}
\label{infty_out}
\end{figure}

We now have a pair of problems which fit into the above program (jumps are $\in L^{\infty} \cap( I + L^2)$).
The proof of existence now comes in three steps: to show $\ID - C_{{\wt V}^ \lra}$ is
Fredholm, has  zero index, and  then that $ {\rm ker}(\ID - C_{{\wt V}^ \lra}) = 0$.
This last point is established through a vanishing lemma similar in spirit to
\cite{DKMVZ99b}, Section $5$.

\subsection{Fredholmness}

Fredholmness is implied by the following  continuity condition holding
throughout the contour.
Moving clockwise about a point $p$ on $\Sigma$, at which segments of the
contour $\Sigma_1$ through $\Sigma_k$ with jumps $v_1$ through
$v_k$ meet, continuity at $p$ is equivalent to
\be
\label{continuity}
    I  = v_1(p)^{\pm 1}  v_2(p)^{\pm 1} \cdots v_k(p)^{\pm 1},
\ee
in which the sign ($\pm 1$) in the exponent is determined by whether the given contour points into,
or out of, $p$.   Additionally, this assessment is invariant of conjugations or deformations, see \cite{Zhou89}.

\subsubsection{Criteria (\ref{continuity}) for $RHP^{\sra}$ }

The conjugation by ${\mathfrak m} (\zeta)$ in the exterior of a large disk not only produced a problem with decay, from
the point of view of Fredholmness, it removed the discontinuity point at $\zeta = \infty$.    While  new point
of intersection $\zeta = p_1, p_2 $ and $p_3$ have been introduced to resulting contour by this move,
each corresponded to a point of continuity for the $( \Sigma^{\sla}, V^{\sla})$ and thus remain so by
the discussion above.

The worrisome points left on $ \wt \Sigma^{\sra}$  are then the origin  and $\zeta = \alpha$.   At the origin,
the jump matrices $\wt \Sigma_1^{\sra} $ through $\wt \Sigma_4^{\sra}$ satisfy,
\[
    I =  \left( \ba{cc} 1 & 1 \\ 0 & 1 \ea \right)^{-1}
           \left( \ba{cc} 1 & 0 \\ 1 & 1 \ea \right)
            \left( \ba{cc} 0 & 1 \\ -1 & 0 \ea \right)
            \left( \ba{cc} 1 & 0 \\ 1 & 1 \ea \right),
\]
and therefore (\ref{continuity}) holds.
At the point  $\zeta=\al$, fix a small $\ep > 0$ $(\ep  < \al)$  and consider the local problem
\[
      (P_{\al}(\zeta) )_+  =  \left\{  \ba{ll}  (P_{\al} (\zeta))_{-}    \left( \ba{cc} 1 & e^{- \frac{4}{3} \zeta^{3/2} } \\ 0 & 1 \ea \right) &
               \mbox{  for }  \zeta \in [\al -\ep, \al]  \\
                       (P_{\al} (\zeta))_{-}         & \mbox{  for } \zeta \in (\al, \al + \ep]  \ea   . \right.
\]
The jump being upper-triangular allows us to write down an explicit solution:
\be
   P_{\al}(\zeta) =  \  \left( \ba{cc} 1 &   \frac{1}{2 \pi i} \int_{\al - \ep}^{\al} \frac{e^{- \frac{4}{3} s^{3/2} }}{s - \zeta}  \, {ds} \\ 0 & 1 \ea \right),
\label{6.8}
\ee
holding in $L^2$ and in the sense of continuous boundary values away from $\zeta = \al - \ep$ and $\zeta = \al$.
(Recall, $C_{+} - C_{-} = \ID$.)
Next choose
a positive  $\ep^{\prime} <  \ep$ and define
\[
     M_{\al}(\zeta ) =  {\widetilde M}^{\sra}(\zeta) \, P_{\al}(\zeta)   \   \  \ \mbox{ for }   | \zeta - \al | < \ep^{\prime},
\]
leaving $  {\widetilde M}^{\sra}(\zeta)$  unchanged in the exterior of this disk.   The effect of conjugating out the
local solution is a new $RHP$  with contour depicted below in Figure \ref{new_inner_contour}.
The point of discontinuity $\zeta = \alpha$ has
been removed, with the introduced point of self-intersection at ${p}=p_4$ again automatically  a continuity point,
having arose from such by way of  a conjugation.

\begin{figure}[h]
\centerline{   
             \scalebox{0.47}{
             \includegraphics{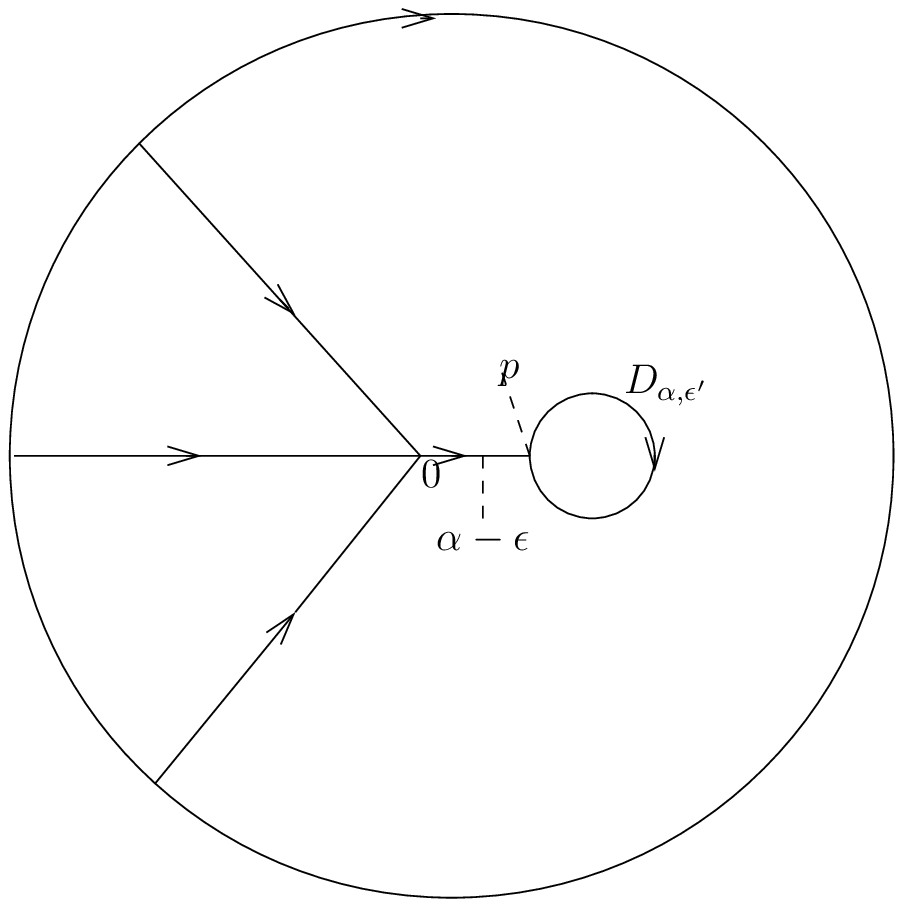}
                }
             }
\caption{The introduced jump along the boundary of the disk $D_{\al, \ep^{\pr}}$ is $P_{\al}^{-1}$.}
\label{new_inner_contour}
\end{figure}

\subsubsection{Criteria (\ref{continuity}) for $RHP^{\sla}$}

From the previous discussion for $RHP^{\sra}$  it is plain that we must only deal with the point
$\zeta = 0$ where the three ray of the contour  $\Sigma^{\sla}$ come together. This
is again handled by conjugating out a local solution local solution $P_0(\zeta)$
of the $RHP$ with jump  $V^{\sla}$ restricted to $\Sigma^{\sla} \cap ( U_{\ep} =  \{ \zeta : | \zeta  | < \ep \})$.
Given $P_{0}$ we set
     $ M_0(\zeta) =  {\wt M}^{\sla} (\zeta)  \,  P_{0}(\zeta)$,    for   $\zeta \in U_{{\ep}^{\pr}}$ with   $\ep^{\pr} < \ep$,
and $M_0 (\zeta) =  {\wt M}^{\sla} (\zeta) $ for $ \zeta \in \C \backslash U_{{\ep}^{\pr}}$.
The $RHP$ for $M_0$
will satisfy (\ref{continuity}),  and it will follow that $( \wt \Sigma^{\sla}, \wt V^{\sla})$ is Fredholm.

To construct $P_0(\zeta)$,  consider 
\[
     \wt M^{\sla}(\zeta) e^{ - (  \frac{2}{3} \zeta^{3/2} + \alpha \zeta^{1/2} ) \sigma_3} =  \wt
     M^{\sla}(\zeta) e^{ - \frac{1}{2} \theta(\zeta) \sigma_3},  \   \   \  \zeta \in U_{\ep},
\]
which has the constant jumps,
\be
\ba{ll}
    e^{   \frac{1}{2} \theta(\zeta) \sigma_3}  \left( \ba{cc} 1 & 0 \\ e^{\theta(\zeta)} & 1 \ea \right)  e^{ -  \frac{1}{2} \theta(\zeta) \sigma_3 }
    =    \left( \ba{cc} 1 & 0 \\ 1 & 1 \ea \right) , &  \zeta \in  (  \Sigma^{\sla} \cap \C_{\pm} )  \cap U_{\ep}, \\
e^{  \frac{1}{2} \theta_{-}(\zeta) \sigma_3}
 \left( \ba{cc} 0 & 1 \\ - 1 & 0 \ea \right)  e^{ -  \frac{1}{2} \theta_{+}(\zeta) \sigma_3}
    =   \left( \ba{cc} 0 & 1 \\ -1 & 0 \ea \right) ,    &     \zeta \in \R_{-} \cap U_{\ep}.
\ea
\ee
Extending the jump contours on the right to
infinity, we obtain the problem: find some
$Q(\zeta)$, which satisfies
\be
\label{firstbessel}
\ba{ll}
 Q_{+}(\zeta)   =   Q_{-} (\zeta)     \left( \ba{cc}  0 & 1 \\ -1 & 0 \ea \right),    & \zeta  \in  \R_{-}, \\
 Q_{+}(\zeta)  =    Q_{-}(\zeta)    \left( \ba{cc}  1 & 0\\ 1 & 1 \ea \right),   &
       \arg \zeta  =  \pm  \frac{2}{3} \pi,
\ea
\ee
with $Q(\zeta)$ otherwise analytic.   From \cite{KMV04}, Section 6,  we have:

\begin{proposition}
With  $H_0^{(1)}(\cdot)$ and $H_0^{(2)}(\cdot)$
denoting the Hankel functions of the first and second kind, and $I_0(\cdot)$ and $K_0(\cdot)$
the usual modified Bessel functions,
the (un-normalized) $RHP$  (\ref{firstbessel}) has the solution:
 \be
\label{Qdef1}
 Q (\zeta) =  \left(  \ba {cc}   I_0 ( \sqrt{\zeta}) &    \frac{i}{\pi}  K_0( \sqrt{\zeta}) \\
                        2 \pi i \sqrt{\zeta}   I_0^{\pr} (\sqrt{\zeta}) &      -2  \sqrt{\zeta} K_0^{\pr} (\sqrt{\zeta})
                                          \ea \right).
\ee
for  $  - \frac{2}{3} \pi < \arg \zeta  <    \frac{2}{3} \pi $,
\be
\label{Qdef2}
Q (\zeta) =  \left(  \ba {cc}  \frac{1}{2} H_0^{(1)}( \sqrt{-\zeta}) &    \frac{1}{2}  H_0^{(2)}( \sqrt{-\zeta}) \\
                        \pi \sqrt{\zeta} \Bigl( H_0^{(1)} \Bigr)^{\pr} (\sqrt{-\zeta}) &      \pi \sqrt{\zeta} \Bigl( H_0^{(2)} \Bigr)^{\pr} (\sqrt{-\zeta})
                                          \ea \right),
\ee
for $  \frac{2}{3}  \pi  <   \arg \zeta  <  \pi $, and
\be
\label{Qdef3}
 Q (\zeta) =  \left(  \ba {cc}  \frac{1}{2} H_0^{(1)}( \sqrt{-\zeta}) &   - \frac{1}{2}  H_0^{(2)}( \sqrt{-\zeta}) \\
                        - \pi \sqrt{\zeta} \Bigl( H_0^{(1)} \Bigr)^{\pr} (\sqrt{-\zeta}) &      \pi \sqrt{\zeta} \Bigl( H_0^{(2)} \Bigr)^{\pr} (\sqrt{-\zeta})
                                          \ea \right),
\ee
 for $     - \pi < \arg \zeta < - \frac{2}{3} \pi $.
 \end{proposition}

It follows that we can set
$
   P_0(\zeta) = Q( \zeta)  \, e^{ \frac{1}{2} \theta(\zeta)  \sigma_3}
$   
to form the needed local solution  in $|\zeta| < \ep$ and complete the proof.

\subsection{Index zero}

One consequence of the Gohberg-Krein theory of factorization of matrix-valued functions,
is that, given Fredholmness, the index of $ \ID - C_V$ equals the winding number of
$\det V$ over the contour, see \cite{Spit}.  But with $V$ equal to either $ \wt V^{\sla}$ or $\wt V^{\sra}$,
$\det V \equiv 1$ and so that winding number is zero.

\subsection{Vanishing Lemma}

Finally we show that $\rm{ ker }( \ID - C_{V^{\lra}} )  = 0$, first reverting back
to the problem(s) tamed at infinity $(\wt \Sigma^{\lra},  \wt V^{\lra})$, recall
(\ref{tame}) and (\ref{tamejumps}).
Suppose that
$\rm{ ker }( \ID - C_{\wt V^{\lra}} ) \neq  0$.  Then there exists $\mu_0^{\lra} \in
L^2(\wt \Sigma) $ which satisfy
\[
    (\ID -   C_{\wt V^{\lra}} ) \mu_0^{\lra} = 0,
\]
and so
\[
    {\wt M}_0^{\lra} (\zeta)   {{=}}   \int_{\wt \Sigma^{\lra}} \frac{ \mu_0^{\lra}(s)  ( I - (\wt V^{\lra}(s))^{-1} )}{ s -\zeta }  \, ds
\]
are $L^2$-solutions of the $RHP$s:
\be
\ba{ll}
 {\wt M}_0^{\lra} (\zeta)  \mbox{ analyltic  in } \C \backslash \wt \Sigma^{\sla}  \mbox{ or } \C    \backslash \wt \Sigma^{\sra}, &  \\
 ({\wt M}_0^{\lra})_{+} (\zeta) =   ( {\wt M}_0^{\lra} )_{-}(\zeta) \, {\wt V}^{\lra}(\zeta),   &   \zeta \in   \wt \Sigma^{\sla}
                  \mbox{ or }  \zeta \in  \wt \Sigma^{\sla}  \\
  {\wt M}_0^{\lra} (\zeta) = \OO \Bigl(  \frac{1}{\zeta} \Bigr) ,  &  \zeta \ra \infty, \   \zeta \notin   \wt \Sigma^{\sla}
                  \mbox{ or }  \zeta \notin  \wt \Sigma^{\sla}  .
\ea
\ee
Given this assessment,  undoing the transformation that took us from the $RHP$s  $( \Sigma^{\lra}, V^{\lra})$,
$(\wt \Sigma^{\lra},  \wt V^{\lra})$,  (removing the conjugation by  $\mathfrak m(\zeta)$ produces
$M_0^{\sra}$ and $M_0^{\sla}$  which solve $RHP^{\sra}$  or
$RHP^{\sla}$ with new asymptotics:
\be
\label{newasymp}
   M_0^{\lra} (\zeta)  = \OO \Bigl( \frac{1}{\zeta} \Bigr)  \zeta^{-\frac{\sigma_3}{4}}  \frac{1}{\sqrt{2}}
   \left( \ba{cc} 1 & 1 \\ -1 &  1 \ea \right) e^{- \frac{i \pi}{4} \sigma_3}  ,   \   \   \  \zeta \ra \infty,
\ee
holding uniformly
in directions respecting
${\rm dist}( \zeta, \Sigma^{\sla}) > \delta$ or ${\rm dist}( \zeta, \Sigma^{\sra}) > \delta$.  We show that
the only conclusion is that $M^{\lra}(\zeta) \equiv 0$.

\subsubsection{Vanishing lemma for $RHP^{\sra}$}

The first step is to fold (and twist) the jumps down to the real line, defining a new matrix
$Z(\zeta)$ via
\beqn
\ba{ll}
   Z(\zeta)   = M_0^{\sra}(\zeta) \,    \left( \ba{cc} 0 & -1 \\ 1 & 0 \ea \right),  &  {\,} 0 < \arg \zeta < \frac{2}{3} \pi     \\
   Z(\zeta)  =   M_0^{\sra}(\zeta) \ \left( \ba{cc} 1 & 0 \\ e^{\frac{4}{3} \zeta^{3/2}} & 1 \ea \right)
                                 \left( \ba{cc} 0 & -1 \\ 1 & 0 \ea \right),   &  \frac{2}{3} \pi <  \arg \zeta < \pi  ,    \\
    Z(\zeta)   =  M_0^{\sra}(\zeta)  \left(  \ba{cc} 1 & 0 \\ - e^{\frac{4}{3} \zeta^{3/2}}  & 1 \ea \right),  &  - \pi < \arg \zeta < - \frac{2}{3} \pi , \\
    Z(\zeta)   =  M_0^{\sra}(\zeta) , &       - \frac{2}{3} \pi < \arg \zeta < 0.
\ea
\eeqn
Then $Z(\zeta)$ is an $L^2$-solution of the equivalent $RHP$:
\beq
\label{foldit1}
\ba{lcll}
   Z_{+} (\zeta) & =  &  Z_{-}(\zeta)  \left( \ba{cc} 1 & - e^{ \frac{4}{3} \zeta_+^{3/2}} \\
                                                                       e^{\frac{4}{3} \zeta_{-}^{3/2}}  & 0 \ea \right) ,  & \zeta \in (-\infty, 0] ,  \\
   Z_{+} (\zeta) & =   & Z_{-}(\zeta) \left(   \ba{cc} e^{- \frac{4}{3} \zeta^{3/2}} & -1 \\ 1 & 0 \ea \right),
                                    &  \zeta \in   (0, \alpha], \\
   Z_{+}(\zeta)  & = &  Z_{-}(\zeta)
                                             \left(\ba{cc} 0 & -1 \\ 1 & 0 \ea \right),
                                     & \zeta \in  (\alpha, \infty),
\ea
\eeq
with now $Z(\zeta) =  \OO(\zeta^{-3/2})$, compare (\ref{newasymp}).
Denoting the piece-wise defined jump matrix  in (\ref{foldit1}) as $V$,
we notice that,
\beq
\label{vanish}
   0  =  \int_{\R}  Z_{+} (s)  Z_{-}^*(s)  \, ds
        =  \int _{\R}  Z_{-} (s) V(s)  Z_{-}^*(s)  \, ds.
\eeq
The first equality holds for any functions in the range of $C_{+}$ and $C_{-}$ as may
be seen by rational approximation. Adding (\ref{vanish}) to its conjugate transpose
we also find that,
\[
 0 = \int_{-\infty}^0  Z_{-} (s)  \left( \ba{cc} 1 & 0 \\ 0 & 0 \ea \right)
                                 Z_{-}^*(s)   \, ds
        +
        \int_0^{\alpha}  Z_{-} (s)  \left( \ba{cc}  e^{- \frac{4}{3} s^{3/2}} & 0 \\ 0 & 0 \ea \right)
             {Z}_{-}^*(s)  \, ds.
\]
It is immediate that the first column of $Z_{-}$ vanishes $a.e.$ on $(-\infty, \alpha]$,
and so $(Z_{11}, Z_{21} ) = 0$ throughout the lower half plane by analyticity
(it lies in the Hardy class $H^2$).  From the structure of the jump across $\R$ one may next
conclude that the second column of $Z_{+}$ equals  $0$ $a.e.$ on $(-\infty, \alpha)$,
and by the same reasoning
$(Z_{12}, Z_{22} ) = 0$ in the upper half plane.

If we now set,
\[
   a(\zeta) = {Z}_{11}(\zeta)   \mbox{ in }  {\C}_{+} \   \   \
   \mbox{ and }  \  \   \  a(\zeta) = {Z}_{12}  (\zeta)  \mbox{ in }  {\C}_{-},
 \]
 we are led to the
 the scalar $RHP$:
 \be
 \label{scalar1}
 \ba{ll}
    a(\zeta) \mbox{ analytic in }  \C \backslash \R_{-}, &   \\
    a_{+}(\zeta) = a_{-}(\zeta) \, e^{  \frac{4}{3} \zeta_{-}^{3/2}} ,   &    \zeta \in  \R_{-},  \\
    a(\zeta) = \OO(\zeta^{-3/4}),  &  \zeta  \in   \C/ \R_{-}.
\ea
\ee
Note that setting $a = Z _{21}$ in $\C_{+}$ and $ =  Z_{22}$ in $\C_{-}$
produces the identical $RHP$.   If we can conclude that $a(\zeta) \equiv 0$,
then $Z(\zeta)$ and so $M_0^{\sra}(\zeta)$ also vanish identically, proving
that ${\rm ker}(\ID - C_{V^{\sla}}) = 0$.

\begin{lemma} The unique solution to (\ref{scalar1}) is $a(\zeta) = 0$.
\end{lemma}

\begin{proof}
Up to this point
the jump condition has been understood in the sense of $L^2$.
To go further it is required that  $a_{\pm}(\zeta)$
are uniformly bounded on the negative real axis.
First, since $e^{ \frac{4}{3} \zeta^3/2}$ is analytic off $\R_{-}$,  one actually has
analytic extensions of $a_{+}$ below $\R_{-}$ and  $a_{-}$ above.
That is, (\ref{scalar1}) holds in the sense of continuous boundary values.
Further,
$e^{\zeta_{-}^{3/2}}$ decays as $\zeta$ moves below $\R_{-}$
and  $e^{ - \zeta_{-}^{3/2}}  = e^{ \zeta_{+}^{3/2}}$ decays as
as $\zeta$ moves above $\R_{-}$, and thus both extensions
exhibit at least the same decay as $a(\zeta)$ itself as $\zeta \ra \infty$.

Taking the extension of $a_{+}(\zeta)$ into a region $ \pi < \arg(\zeta) <  \pi + \ep$,  the
Cauchy integral formula provides the representation
\be
   a_{+}(\zeta) = \int_{\mathfrak C} \frac{\tilde{a}(z)}{z - \zeta} \frac{dw}{2 \pi i}  ,   \   \  \  \   \    \zeta \in \R_{-},
\ee
in which
\[
    {\mathfrak C} =  \{ z : \arg(z) =  -\pi + \ep/2  \}  \cup   \{ z : \arg(z) = \pi - \ep/2  \}  =   {\mathfrak C}^{-} \cup {\mathfrak C}^{+},
\]
oriented counter-clockwise, and $\tilde{a}(z) = a(z), a(z) e^{-z^{3/2}} $ on ${\mathfrak C}^{-}, {\mathfrak C}^{+}$.
It follows that for all $\zeta \in (-\infty, \delta]$  $|a_{+}(\zeta)|$ is
bounded by a constant depending only on  $\delta > 0$ .    An identical argument pertains to $a_{-}(\zeta)$.
To achieve a bound down to $\zeta = 0$, we need
only note that the jump $ = e^{\zeta_{-}^{3/2}}$ for $\zeta < 0$ and $ = 1$ for $\zeta \ge 0$ is
 H\"older continuous across zero and  the Cauchy
transform preserves H\"older continuity.

Granted that $a(\zeta)$ is  bounded down to $\R_{-}$ from both directions, consider now
the effect of performing both the above extensions: $a_{+}$ below to an angle $ \pi +  \pi \nu/2$
and $a_{-}$ above to an angle $-\pi - \pi \nu /2$ with small $\nu > 0$ .  The resulting function, denoted by $b(\zeta)$,
can be made to live on a subset $\mathbb A$ of the Riemann surface ${\mathbb K}$ formed by gluing together
three copies of $\C$ cut across $\R_{-}$ in the obvious fashion (alleviating the fact
that the initial domain swept out a region of angle $> 2 \pi$).

Next bring in the transformation $\zeta(\omega) = \omega^{ 2 + \nu}$
which
maps the right half of the $\omega$-plane
onto $\mathbb A$, taking the positive/negative imaginary axes  in $\omega$ onto the lower/upper boundaries
of $\mathbb A$.  Then
$
      {\hat b}(\omega) = b (\zeta(\omega))
$
is analytic in the open half-plane $ \{ \omega : \Re(\omega) > 0 \}$,
bounded in the closed half-plane $ \{ \omega : \Re(\omega) \ge 0 \}$,
and along the boundary satisfies, 
$
  | b(i s) |  \le  C e^{- c |s|^{3/2 ( 2 + \eta)}}  \le  \tilde{C} \, e^{-  \tilde{c} |s|}.
$
A theorem of Carlson (\cite{ReedSimon4} p. 236) then implies that $\hat b \equiv 0$
in the right half plane, which is to say that $ a \equiv 0$.
\end{proof}

\subsubsection{Vanishing lemma for $RHP^{\sla}$}

The steps for $RHP_{\ra}$ are mimicked to the point  that the needed conclusion
hinges on the following.

\begin{lemma}  The  unique solution of
the scalar $RHP$,
\be
\label{scalar2}
 \ba{ll}
    a(\zeta) {\mbox{ analytic in }}  \C \backslash \R_{-}, &   \\
    a_{+}(\zeta) = a_{-}(\zeta) \, e^{ \frac{4}{3} \zeta_{-}^{3/2} +2  \alpha \zeta_{-}^{1/2}} ,   &    \zeta \in  \R_{-},  \\
    a(\zeta) = \OO(\zeta^{-3/4}),  &  \zeta  \in   \C/ \R_{-},
\ea
\ee
is $a \equiv 0$.
\end{lemma}

\begin{proof}
The analysis of (\ref{scalar2}) is really no different than (\ref{scalar1}).
The jump is again analytic in $\C/\R_{-}$, implying that $a$ is continuous and uniformly bounded
down to  $\R_{-}$ by the same type of extension argument.
Since $\alpha > 0$,
\[
    | e^{ \frac{4}{3}\zeta^{3/2} - 2 |\alpha| \zeta^{1/2}  } |  < 1  \   \  \ \mbox{ for } \{ \zeta:  \arg(\zeta) \in (\pi - \ep, \pi)  \cup ( -\pi, -\pi + \ep  )  \}
\]
and
\[
   | e^{ \frac{4}{3} \zeta^{3/2} - 2  |\alpha |\zeta^{1/2}  } | \le e^{- c s^{3/2}}  
   \  \  \ \mbox{ for } c > 0 \mbox{ and }  \zeta  = se^{i ( \pm \pi \mp \ep)} ,  \  s  \ra \infty,
\]
thus extensions share the same qualitative features as above.
The analogs of $b$ and $\hat b$ are then constructed as before and subject to the same conclusions.
\end{proof}

\section{Properties of the solutions}
\setcounter{num}{7}
\setcounter{equation}{0}
\label{properties}

We prove a continuity result for $M^{\sra}(\zeta)$ and $M^{\sla}(\zeta)$ in the parameter $\alpha$
and
establish asymptotics of those matrix functions for $\alpha \ra \pm \infty$; these will
lead  to Theorems \ref{contthm} and \ref{asympcor}.

\subsection{Continuity}

The continuity result is based on verifying the condition of the following general fact;
see Corollary $7.103$ of \cite{Deift99} for a proof.

\begin{proposition}
\label{contprop}
 Consider a family of (uniquely solvable) $RHP$'s  on a fixed contour, $(\Sigma, v_{n})$,
$n = 1, 2, \dots$.
Assume the existence of a  $v_{\infty}$,  such that the $RHP$ $(\Sigma, v_{\infty})$ posseses a unique
solution and
\be
\label{est1}
    ||  v_n - v_{\infty}  ||_{L^{\infty}(\Sigma)  \cap L^2(\Sigma)}  \, \ra 0,  \  \  \   \  n \ra \infty.
\ee
Then,
\be
\label{conc1}
      ||  (m_n)_{\pm} - (m_{\infty})_{\pm} ||_{L^2(\Sigma)}  \ra 0,   \  \
\mbox{ and }
 \  \
     ||  (m_n)(z) - (m_{\infty})  (z) ||_{L^{\infty}(\mathcal A)}    \ra 0,
\ee
for $n \ra \infty$ and any set $\mathcal A $ which is a positive distance from $ \Sigma.$
 \end{proposition}

The condition (\ref{est1}) implies that $\mu_{v_n} = (\ID - C_{v_n})^{-1} I$ satisfies
$ || \mu_{v_n} - \mu_{v_{\infty}} ||_{L^2(\Sigma)} \ra \infty$.  From the expression (\ref{solrep}),
the statements of (\ref{conc1}) easily follow; the first because $C_{\pm}$ map $L^2$ to $L^2$.
Further, one sees that an estimate of the second type holds for the derivatives,
$\frac{d}{dz} [  (m_n)(z) - (m_{\infty}) (z) ]$.  This is the fact referred to in the proof of Theorem \ref{maintheorem}.

\begin{lemma}
\label{contlemma}
The condition (\ref{est1}) is satisfied by the  $RHP$'s  $(V^{\lra}, \Sigma^{\lra})$,
the parameter $\alpha$ playing the role of $n$ in the Proposition.  Continuity  holds
in each problem down to $\alpha = 0$.
\end{lemma}

\begin{remark}  Note that the problem $RHP^{\sra}$ has $\alpha$-dependence
in the contour itself via the segment $[ 0, \alpha)$.   In this case, for the $L^2$ continuity of the boundary values
 $(M^{\sra})_{\pm} = (M_{\alpha}^{\sra})_{\pm}$,  we show continuity in $L^{2}( \wh \Sigma^{\sra})$
where in $\wh \Sigma^{\sra}$, the segment $[0,\alpha)$ is extended to  $[0, \alpha^{\pr})$  for any $\alpha^{\pr} > \alpha$.
\end{remark}

\begin{proof}[Proof of Lemma \ref{contlemma} for $RHP^{\sra}$]
To employ the conditions of Proposition \ref{contprop} a preliminary conjugation is
made to move the dependence of the problem on $\alpha \ge 0$ from the contour
into the jump.

Consider first the continuity at a point  $\alpha > 0$. In this case the conjugation
is  affected by the same parametrices used in the poof of existence.   Set aside a neighborhood
of $\alpha$,  $U_{\alpha, \ep} = (\alpha - \ep, \alpha + \ep)$  for $\ep > 0$   with $\ep \ll \alpha$.  Fix also
positive $a$ and $b$ with $a <  b < \alpha - \ep$ and disks $D_{a}$, $D_b$ enclosing
$\alpha + \ep$ ($D_b \subset D_a$).
Within $D_a$, and for any $\beta \in U_{\alpha, \ep}$, we have the parametrices,
\[
  P_{\beta}(\zeta)  = \left( \ba{cc}     1   &    \frac{1}{2 \pi i} \int_a^{\beta}   \frac{e^{- \frac{4}{3}s^{3/2}}}{ s - \zeta} ds \\
                                                                 0  &   1 \ea       \right)  \equiv \left(  \ba{cc} 1 & C_{\beta}(\zeta) \\ 0 & 1   \ea \right).
\]
That is, $P_{\beta}(\zeta)$ satisfies the jump condition across $a < \zeta < \beta$.  Next define,
\be
 \label{inDb}
      \wt M_{\beta}(\zeta) =  \left\{  \ba{ll}   M^{\sra}(\zeta) \, P_{\beta}(\zeta),  &  \zeta  \in D_b, \\
                                                      M^{\sra}(\zeta), & \zeta \in \C \backslash D_b. \ea \right.
\ee
Now for all $\beta$ in the defined  range we have a family of $RHP$'s on the same contour, with the dependence
of $\beta$ occurring only in the jump
\[
  \wt V_{\beta}(\zeta) =  P_{\beta}^{-1}(\zeta),   \   \  \  \mbox{ for }  \zeta \in   \partial D_b.
\]
Also, for all $\zeta \in \partial D_b$  except $\zeta = b$,
\be
\label{contestimate}
  | C_{\beta} (\zeta) -  C_{\alpha}(\zeta) | \le     \Bigl|  \int_{\beta}^{\alpha}  \frac{e^{-4/3 s^{3/2} }}{s - \zeta}  ds  \Bigr|   = \OO(\beta- \alpha),
\ee
there being a positive distance $\zeta$ separating  and the interval between $\alpha$ and  $\beta$ in this case.  On the other hand $(C_{\beta})_{\pm}(b)
= (C_{\alpha})_{\pm}(b)$.   Thus, the above  $L^{\infty}$  estimate
leads to an $L^2(\partial D_b)$ of the same order.

It  follows see that $\wt M_{\beta}$ satisfies the criteria of Proposition \ref{contprop},
and so $M^{\sra}(\zeta)$ is continuous at any $\alpha > 0$ in the sense of
its boundary values in $L^2( \Sigma^{\sra} \cap (\C \backslash D_b)$,
and also in $L^{\infty}$ for $\zeta$ exterior to $D_b$ and away from $\Sigma^{\sra}$.
This already gives the type of continuity claimed in Theorem \ref{contthm} and Claim \ref{claim2}
used in the proof of Theorem \ref{maintheorem}.

We complete the analysis by showing the boundary data is $L^2$-continuous in the interior
 $D_b$. Inverting the move (\ref{inDb}) we have.
\beq
\label{backin}
 M^{\sra}(\zeta)  & = &  \wt M_{\beta}(\zeta)  P_{\beta}^{-1}(\zeta) \\
                  & = &   - \wt M_{\beta} (\zeta) \, C_{\beta}(\zeta) \,   \left( \ba{cc}   0 &   1  \\ 0 & 0  \ea \right)
                                +   \wt M_{\beta}(\zeta).  \nonumber
\eeq
Consider the  $\pm$-limits of the right hand: we want to show they are continuous
in $L^2[b, \alpha + \ep]$  as $\beta$ ranges in $U_{\alpha, \ep}$.
First,
$\wt M_{\beta}(\zeta)$  is analytic inside of $D_b$ with continuous boundary values along $\partial D_b$,
excepting the point $\zeta = b$.
It therefore lies in  $L^{\infty}$ of that interval, and  the conclusions
above include that $\beta \ra \wt M_{\beta}(\zeta)$ is continuous in $L^2[b, \alpha+ \ep]$.
A look at the second line of (\ref{backin}) explains that it remains to show
that $(C_{\beta})_{\pm}$ are continuous in $L^2[b, \alpha+\ep]$.  But, taking
$\beta \downarrow \alpha$ from above without any loss of generality,
\beqn
\label{L2first}
  || (C_{\beta})_{\pm}  - (C_{\alpha})_{\pm}    ||_{L^2[b, \alpha+\ep]}^2   & = &
      \int_{\alpha}^{\beta}   e^{-\frac{8}{3}s^{3/2}}  \,  ds
      +  \frac{1}{4 \pi^2} \int_{\beta}^{\alpha+ \ep}    \left|  \int_{\alpha}^{\beta}  \frac{e^{-\frac{4}{3}s^{3/2}}}{s - t}  \, ds \right|^2   dt  \,  \ra \, 0.
\eeqn
Here, the H\"older continuity of $e^{-4/3 s^{3/2}}$  produces the vanishing of the integral over $[b, \alpha]$,
and  the first term on the right follows from $ || C_{\pm} \diamond ||_{L_2} \le || \diamond ||_{L^2}$.
This completes the proof for $\alpha > 0$.

Turning to the case $\alpha = 0$, the first step is to delay
the jumps to the left of the origin in the original problem
$( V^{\sra}, \Sigma^{\sra})$ by considering the equivalent $RHP$:
\[
\ba{ll}
      \wh M_+(\zeta) = \wh M_{-}(\zeta)  \left( \ba{cc}   1 & e^{- \frac{4}{3} \zeta^{3/2}} \\ 0 & 1 \ea  \right),  &  \zeta \in [0, \beta],  \\
     \wh M_+(\zeta) = \wh M_{-}(\zeta)  \left( \ba{cc}   e^{ \frac{4}{3}{\zeta}_{+}^{3/2}}  & 1
                                                                     \\ 0 & e^{{- \frac{4}{3}\zeta}_{-}^{3/2}}  \ea  \right),  &  \zeta \in [-1, 0],  \\
        \wh M_+(\zeta) = \wh M_{-}(\zeta)  \left( \ba{cc}   0  & 1
                                                                     \\ -1 & 0  \ea  \right),  &  \zeta \in (-\infty, -1),      \\
             \wh M_+(\zeta) = \wh M_{-}(\zeta)  \left( \ba{cc}   1  & 0
                                                                     \\ e^{
                                                                      \frac{4}{3}{\zeta}^{3/2}}  & 1 \ea  \right),  &
                                                                       \zeta \in \{ \zeta :  \arg(1 + \zeta)    = \pm \frac{2}{3} \pi  \},
\ea
\]
where $\wh M(\zeta)$ is otherwise analytic and equals $(I + \OO(\zeta^{-1}) ) {\mathfrak m}(\zeta)$  as $\zeta \ra \infty$
(recall (\ref{twistsol})).
This problem is obtained from $(V^{\sra}, \Sigma^{\sra})$ by setting
\[
   \wh M(\zeta ) = M^{\sra} (\zeta) \left( \ba{cc}  1 & 0 \\  \pm e^{\frac{4}{3} \zeta^{3/2} } & 1 \ea \right),  \  \   \  \  \
   \zeta \in {\mathcal W}_{\pm},
\]
where $\mathcal W_{\pm}$  is the intersection of $
\C_{+}$ or $
\C_{-}$ with the
region bounded between the rays
\[
    \{ \zeta :  \arg(1 + \zeta)    = \pm \frac{2}{3} \pi \}  \  \mbox{ and }  \ \{ \zeta :  \arg(\zeta)    = \pm \frac{2}{3} \pi \}.
\]
Proving we have $L^2$ continuity here for $\beta \downarrow 0$ will imply the
same for the original problem.

Similar to above, we now set
\be
\label{secondP}
P_{\beta}(\zeta) = \left( \ba{cc} 1 &
                                                    \frac{1}{2 \pi i} \int_{-1}^{\beta} \frac{f(s)}{s - \zeta} ds \\ 0 & 1 \ea \right)
                             \equiv     \left( \ba{cc} 1 & C_{\beta}(\zeta) \\ 0       & 1 \ea \right),
\ee
where
\be
\label{seconexp}
 f(s) =   1  \mbox{ for }   -1 < s < 0  \    \mbox{   and    }  \    f(s)   =  e^{- \frac{4}{3} s^{3/2}} \mbox{ for }  0 \le s < \beta.
\ee
Again, the point is that  $P_{\beta}(\zeta)$ satisfies the jump condition
\[
   (P_{\beta})_+(\zeta) = (P_{\beta})_{-}(\zeta)  \, \left( \ba{cc}   1 &  f(\zeta) \\ 0  & 1 \ea  \right),  \   \   \  \zeta \in (-1, \beta).
\]
Conjugating out by $P_{\beta}(\zeta)$ inside a disk $D_{1/2} = \{ \zeta:  |z - 1/2| < 1/2 \}$ has three affects.
First, a new jump of $P_{\beta}^{-1}(\zeta)$ is produced along $\partial D_{1/2}$.  Second, the jump which $\wh M(\zeta)$
has across $[0, \beta]$ is eliminated.  Third, the jump across $[-1/2, 0]$ now reads
\beq
\label{jumpconj}
\lefteqn{ \hspace{-2cm}
     \left( \ba{cc}   1  & (C_{\beta})_{-}(\zeta) \\ 0 & 1 \ea  \right)  \left( \ba{cc}   e^{\frac{4}{3} \zeta_+^{3/2}} & 1
                \\ 0 &  e^{-\frac{4}{3} \zeta_-^{3/2}}  \ea  \right)
     \left( \ba{cc}   1 &  -(C_{\beta})_+(\zeta) \\ 0 & 1 \ea  \right) } \\
     &  &  \   \   \   \  =  \left(  \ba{cc}    e^{\frac{4}{3} \zeta_-^{3/2}}  &
     1 + (C_{\beta})_{-}(\zeta) e^{\frac{4}{3} \zeta_{-}^{3/2}}  - (C_{\beta})_{+}(\zeta)  e^{-\frac{4}{3} \zeta_{+}^{3/2}}         \\
                                  0  & e^{\frac{4}{3} \zeta_-^{3/2}} \ea  \right).  \nonumber
\eeq
We already understand that the jump $P_{\beta}^{-1}(\zeta)$ is continuous in $L^{\infty} \cap L^2$ of $\partial D_{1/2}$.  To check
that the jump (\ref{jumpconj}) satisfies the like conditions over $-1/2 \le \zeta \le 0$, note that it is only the $(2,1)$-entry
which requires investigation and that term (neglecting the constant $1$) may be rewritten as in,
\beq
\label{21term}
 \lefteqn{  (C_{\beta})_{-}(\zeta) e^{\frac{4}{3} \zeta_{-}^{3/2}}  - (C_{\beta})_{+}(\zeta)  e^{-\frac{4}{3} \zeta_{+}^{3/2}}  } \\
   & = &   (C_{\beta})_{-}(\zeta)  -  (C_{\beta})_{+}(\zeta) +    (C_{\beta})_{-}(\zeta)  \left(e^{\frac{4}{3} \zeta_{-}^{3/2}} - 1 \right)
                            -    (C_{\beta})_{+}(\zeta) \left( e^{-\frac{4}{3} \zeta_{+}^{3/2}} - 1 \right)  \nonumber \\
   & = & -1 +     (C_{\beta})_{-}(\zeta)  \left(e^{\frac{4}{3} \zeta_{-}^{3/2}} - 1 \right)
                            -    (C_{\beta})_{+}(\zeta) \left( e^{-\frac{4}{3} \zeta_{+}^{3/2}} - 1 \right).            \nonumber
\eeq
The continuity in $L^2[-1/2,0]$ as $\beta \downarrow 0$ follows by a computation similar to (\ref{L2first}) and
the boundedness of $ e^{\pm\frac{4}{3} \zeta_{\pm}^{3/2}} - 1$.   As for the continuity in $L^{\infty}$ recall that
the maps $C_{\pm}$ maintain  H\"older continuity, so there is no problem for $\zeta$ in the interior of $[-1/2, 0]$.
The potential issue of the logarithmic singularity of $ (C_{\beta})_{\pm}(0) $ as $\beta \downarrow 0$ is
countered by the fact that $e^{\mp\frac{4}{3} \zeta_{\mp}^{3/2}} - 1$ vanish to higher order at the origin.

The criteria (\ref{est1}) has thus been checked for the new $RHP$ created by the  conjugation by $P_{\beta}(\zeta)$
defined in (\ref{secondP}) and (\ref{seconexp}) within a neighborhood $D_{1/2}$ of the origin.  It remains to invert this
move and show that $L^2$-continuity at $\alpha  = 0$ of $\wh M_{\pm}$ (and so $M^{\sra}_{\pm}$) follows suit.
However, the  needed argument is identical to that given above in (\ref{backin})
and surrounding discussion.
\end{proof}

\begin{proof}[Proof of Lemma \ref{contlemma} for $RHP^{\sla}$]
The verification of the conditions in this case is straightforward on account of the
contour $\Sigma^{\sla}$ being independent of $\alpha$ from the start.
 For any positive $\alpha$ and $\beta$,
 the difference of $V_{\beta}^{\sla}$ and  $V_{\alpha}^{\sla}$  of course vanishes on $R_{-}$, while on
 the lines $\gamma_{\pm} = \{ \zeta:  \arg(\zeta) = \pm \frac{2}{3} \pi \}$,
 \[
   (V_{\beta}^{\sla} - V_{\alpha}^{\sla})(\zeta) =  \left(   \ba{cc}   0  & 0 \\
                      e^{\frac{4}{3} \zeta^{3/2}} ( e^{2\beta \zeta^{1/2} } - e^{2 \alpha \zeta^{1/2} } ) & 0   \ea\right).
 \]
Due to the decay of $e^{4/3 \zeta^{3/2}}$ along $\gamma^{\pm}$ it is plain that
\[
   ||  e^{\frac{4}{3} \zeta^{3/2}} ( e^{2 \beta \zeta^{1/2} } - e^{ 2 \alpha \zeta^{1/2} } ) ||_{L^{\infty}(\gamma^{\pm}) \cap L^2(\gamma^{\pm}) } \ra 0,
\]
as $\beta \ra \alpha$, the case of $\alpha = 0$ and $\beta \downarrow 0$  being no different.
\end{proof}

\subsection{Asymptotics as $\alpha \ra \pm \infty$}

As  $\alpha \ra + \infty$, it is intuitive that the (unique) solution of $RHP_{\ra}$ should converge to the
solution $P_{\rm A}(\zeta)$ of the $RHP$ defined by the jump conditions,
\be
\label{airy}
\ba{ll}
   (P_{\rm A})_{+}(\zeta) = (P_{\rm A})_{-}(\zeta) \left( \ba{cc}  1  & e^{- \frac{4}{3} \zeta^{3/2} }  \\ 0 & 1 \ea \right),   &   \zeta   \in R_{+} , \\
  ( P_{\rm A})_{+}(\zeta) = (P_{\rm A})_{-}(\zeta) \left( \ba{cc}  1  & 0\\ e^{\frac{4}{3} \zeta^{3/2} } &  1 \ea \right),  &
   \arg \zeta  = \pm  \frac{2}{3} \pi,   \\
  ( P_{\rm A})_{+}(\zeta) = P_{-}(\zeta) \left( \ba{cc}  0  & 1\\ -1  & 0 \ea \right),   &   \zeta   \in R_{-} , \\
\ea
\ee
with $P(\zeta)$ having the same asymptotics as $M^{\ra}(\zeta)$ as $\zeta \ra \infty$.  As is well known,
$P(\zeta)$ is given explicitly in terms of  the Airy function ${\Ai}(\zeta)$ and its derivative .
In particular,  with $\om = e^{\frac{2}{3} \pi i}$, let
\be
\ba{lll}
     P(\zeta) & =     \left( \ba{cc}  \Ai(\zeta)  &  \Ai(\om^2 \zeta)  \\ \Ai^{\pr}(\zeta)  & \Ai^{\pr}( \om^2 \zeta)  \ea \right), &
          \zeta \in \C_{+} ,  \\
     P(\zeta) & =      \left( \ba{cc}  \Ai(\zeta)  &    - \om^2 \Ai(\om^2 \zeta)  \\  \Ai^{\pr}(\zeta)  & - \Ai^{\pr}( \om^2 \zeta)  \ea \right),
          & \zeta \in \C_{-} ,
         \ea
\ee
and $\Upsilon(\zeta) =     \left( \ba{cc} 1 & 0 \\ e^{\frac{4}{3} \zeta^{3/2}} & 1 \ea \right)  $.   Then,
\be
\ba{lc}
   P_{\rm A} (\zeta)  = \sqrt{ 2 \pi} e^{ -\pi i / 12}  \, P(\zeta) \,  e^{( \frac{2}{3} \zeta^{3/2} - \frac{\pi i}{6} ) \sigma_3},   &
    - \frac{2}{3} \pi <      \arg \zeta   < \frac{2}{3} \pi, \\
   P_{\rm A} (\zeta)  = \sqrt{ 2 \pi} e^{ -\pi i / 12}  \, P(\zeta) \,  e^{( \frac{2}{3} \zeta^{3/2} - \frac{\pi i}{6} ) \sigma_3} \,
      \Upsilon(\zeta)^{-1}  ,   &
        \frac{2}{3} \pi < \arg \zeta  < \pi,   \\
    P_{\rm A} (\zeta)  = \sqrt{ 2 \pi} e^{ -\pi i / 12}  \, P(\zeta) \,  e^{( \frac{2}{3} \zeta^{3/2} - \frac{\pi i}{6} ) \sigma_3}
       \, \Upsilon(\zeta)  ,   &
      - \pi < \arg \zeta   <  -\frac{2}{3} \pi
\ea
\ee
We have the following.

\begin{lemma}
\label{airyasymp}
As $\alpha \ra + \infty$,
\be
   M^{\sra}(\zeta)  (P_{\rm A})^{-1}(\zeta)  =  \Bigl( I +  {\OO} (e^{- \frac{2}{3} \alpha^{3/2}} ) \Bigr),
\ee
uniformly for $\zeta$ supported away from $\C \backslash [\alpha, \infty)$.
\end{lemma}

If instead $\alpha \ra - \infty$, one takes advantage of two facts.  First,   the
jump for  $RHP^{\sla}$ along $\arg \zeta = \pm \frac{2}{3} \pi$ satisfies
\[
    \left( \ba{cc}  1 &  0 \\ e^{ \frac{4}{3} \zeta^{3/2}  + 2  \alpha \zeta^{1/2} } & 1  \ea \right)
    =   \left( \ba{cc}  1 &  0 \\ e^{  - 2  |\alpha| \zeta^{1/2} (1 + o(1))}  & 1  \ea \right),    \mbox{ for  } \zeta =  { o}(|\alpha|),
\]
and, second, the unique solution of the $RHP$: $P_{\rm B}(\zeta)$ analytic in $\C \backslash \Sigma^{\sra}$,
\be
\label{Bjumps}
\ba{ll}
  (P_{\rm B})_{+} (\zeta) =  (P_{\rm B})_{-} (\zeta) \left( \ba{cc}  1 &  0 \\ e^{  -  2 \zeta^{1/2} }  & 1  \ea \right), &
    \zeta \in \gamma^{\pm} \\
  (P_{\rm B})_{+} (\zeta) =  (P_{\rm B})_{-} (\zeta) \left( \ba{cc}  0 &  1 \\   { - 1}  & 0 \ea \right), &    \zeta \in (-\infty, 0),
\ea
\ee
and
\be
\label{Basym}
 \hspace{1cm}  P_{\rm B} (\zeta) =    \zeta^{- \sigma_3 /4}  \frac{1}{\sqrt{2}}
      \left(  \ba{cc}  1 &  i \\ i  & 1 \ea  \right)
    \Bigl( I + \OO ( {| \zeta|^{-1/2}} ) \Bigr),  \  \   \   \   \     \zeta \ra \infty  ,
\ee
is known explicitly in terms of Hankel functions.  Here $\gamma^{\pm}$ are any rays  (eventually straight) rays
extending above and below the negative real axis as in Figure \ref{besselcurve}.
 We have in fact already seen the solution in part.
 Set  $Q(\zeta)$ to be as defined in $(\ref{Qdef1})$, $(\ref{Qdef2})$
and $(\ref{Qdef3})$ but in regions $I$, $II$, and $III$ respectively (see again Figure \ref{besselcurve}).
Then,
\be
\label{PBdef}
   P_{\rm B} (\zeta)  = \sqrt{2 \pi} \, Q( \zeta) \,  e^{ - {\zeta}^{1/2} \sigma_3 }.
\ee
That (\ref{PBdef}) satisfies the jumps (\ref{Bjumps}) is immediate from the jump relations for
$Q(\zeta)$. Note that replacing the straight lines $\arg \zeta = \pm 2 \pi/3$ with $\zeta \in \gamma^{\pm}$
has no affect: $Q(\zeta)$ is analytic off $\R^{-}$ and the jump contours may be deformed to accomodate
this change. Lastly, the asymptotics  (\ref{Basym}) can be  verified from substituting the formulas,
\[
  H_{0}^{(1)}(\zeta) = \sqrt{ \frac{2}{ \pi \zeta} } e^{ i (\zeta - \frac{\pi}{4})} ( 1 + \OO( \zeta^{-1/2} ) )
  ,  \  \  \  \   H_{0}^{(2)}(\zeta) = \sqrt{ \frac{2}{ \pi \zeta} } e^{ -i (\zeta - \frac{\pi}{4})} ( 1 + \OO( \zeta^{-1/2} ) ),
\]
for $\zeta \ra \infty$ (\cite{AS}, formulas 9.7.1 - 9.7.4) into the definition of $ Q(\zeta)$.


\begin{figure}
\centerline{   
            \scalebox{0.65}{            \includegraphics{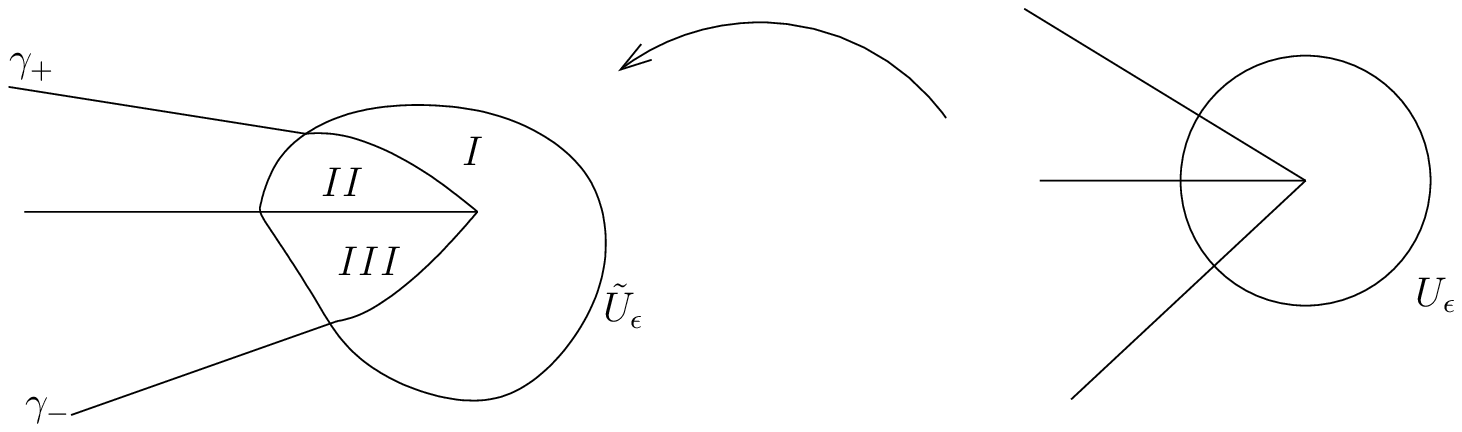}
              }
            }

\caption{The contour for $P_{\rm B}$ and corresponding transformation.}
\label{besselcurve}
\end{figure}

The analogue of Lemma \ref{airyasymp} can now be stated.

\begin{lemma}
\label{besselasymp}
Set
\be
\label{EEEdef}
    E_{\alpha}(\zeta) =   \Bigl(  |\alpha| - \frac{2}{3} \zeta  \Bigr)^{\sigma_3/2}  e^{-\frac{ \pi i}{4} \sigma_3},  \  \   \mbox{ for }  |\zeta| < \ep |\alpha|,
\ee
with any $\ep < 1$.  Then
$\gamma^{\pm}$ in (\ref{Bjumps}) may be chosen in such a way that
\be
    M^{\sla}(\zeta) \left[ E_{\alpha}(\zeta) \,
     P_{\rm B} \left(  \zeta  ( |\alpha | - \frac{2}{3} \zeta)^2 \right) \right]^{-1} =   \Bigl( I +  {\OO}  ( {|\alpha|}^{-1} ) \Bigr),
\  \  \  \alpha \ra - \infty,
\ee
uniformly on $|\zeta| < \ep |\alpha|$.
\end{lemma}

From Lemmas \ref{airyasymp} and \ref{besselasymp}, Corollary \ref{asympcor} is read off
immediately from the explicit forms of  $P_{\rm A}$ or $E_{\alpha} P_{\rm B}$: we have for instance,
\[
   (M^{\sra})_{11}(\zeta) =  \sqrt{ 2\pi }  e^{- \pi i /4}  \Ai(\zeta)  e^{\frac{2}{3} \zeta^{3/2} } \, (1 + \OO( e^{- \frac{2}{3} \alpha^{3/2}} ) )
\]
and
\[
   (M^{\sla})_{11}(\zeta)
    =  \sqrt{ 2\pi }  e^{- \pi i /4}  ({ |\alpha| - \frac{2}{3} \zeta})^{1/2}  \, I_{0} \Bigl( \sqrt{\zeta} ( |\alpha| - \frac{2}{3} \zeta)  \Bigr)
     e^{ -  \sqrt{\zeta} ( |\alpha| - \frac{2}{3} \zeta ) } \, (1 + \OO( |\alpha|^{-1} ) ).
\]
Lemmas \ref{airyasymp} and \ref{besselasymp} themselves follow directly from checking
condition (\ref{est2}) of the below proposition, the proof of which may be found in \cite{DKMVZ99a}, Section 7.

\begin{proposition}
\label{neumann}
If for a family of $L^2$-solvable $RHP$'s  $(\Sigma, v_{n})$ there is
the estimate
\be
\label{est2}
   || v_n - I ||_{L^{\infty}(\Sigma) \cap L^2(\Sigma)}  \le \frac{C}{n},
\ee
for a fixed constant $C$ and all large $n$, then
$ || C_{v_n}  ||_{L^2(\Sigma) \ra L^2(\Sigma)}  =    \OO( \frac{1}{n})$,
and the solutions satisfy  $
    m_{n}(\zeta) = I + \OO( \frac{1}{n}) $
 uniformly for $\zeta$ a positive distance  from   $\Sigma$.
\end{proposition}

The fact that (\ref{est2}) implies a like bound on the operator norm of $C_{v_n}$ actually
implies the existence of a full asymptotic expansion of $M^{\sra}(\zeta)$ and $M^{\sla}(\zeta)$
in powers of $e^{- \alpha^{3/2}}$ or $ \alpha^{-1}$ with sectionally analytic coefficients.
This is not pursued here.

\begin{proof}[Proof of Lemma \ref{airyasymp}]
Define
\[
   R(\zeta) = M^{\sra}(\zeta) (P_{\rm A})^{-1}(\zeta),
\]
which solves the $RHP$:
\be
\ba{ll}
  R(\zeta)  \mbox{ analytic in }  \C \backslash [ \alpha, \infty)  & \\
  R_{+}(\zeta)     = R_{-}(\zeta)  \, \left( (P_{\rm A})_{-}(\zeta) \, \left( \ba{cc} 1 & - e^{-\frac{2}{3} \zeta^{3/2}} \\ 0 & 1 \ea \right)  \,
                                                                   (P_{\rm A})_{-}^{-1}(\zeta) \right),  &
   \zeta \in [ \alpha, \infty) \\
   R(\zeta) = I + O(\frac{1}{\zeta}),  & \zeta \ra \infty.
\ea
\ee
The jump matrix along $[\alpha, \infty)$ can be simplified as in
\beqn
 (P_{\rm A})_{-}(\zeta) \, \left( \ba{cc} 1 & - e^{-\frac{2}{3} \zeta^{3/2}} \\ 0 & 1 \ea \right)  \, (P_{\rm A})_{-}^{-1}(\zeta)
    & = &  I  -  e^{-\frac{\pi i}{3} }  \left( \ba{cc}  - \Ai(\zeta) \Ai^{\pr}(\zeta)   &   \Ai^{2}(\zeta) \\
                                                                                                           ( \Ai^{\pr}(\zeta)  )^2  &   \Ai(\zeta) \Ai^{\pr}(\zeta)  \ea \right) \\
    & \equiv & I -  V_{R} (\zeta).
\eeqn
Next, noting the asymptotics,
\[
\ba{ll}
   \Ai(\zeta)           =  &  \frac{z^{-1/4}}{ 2 \sqrt{\pi}}  e^{-\frac{2}{3} \zeta^{3/2} } ( 1 + \OO( |\zeta|^{-3/2})),  \\
   \Ai^{\pr}(\zeta)  =  &     \frac{-z^{-1/4}}{ 2 \sqrt{\pi}}  e^{-\frac{2}{3} \zeta^{3/2} } ( 1 + \OO( |\zeta|^{-3/2})) ,
 \ea    \ \  \   \zeta \ra \infty,  \  \ \ |\arg(\zeta) | \le  \frac{2}{3} \pi,
\]
(see \cite{AS}, p. 446),
we have that both
$ || V_R ||_{L^{\infty}[\alpha, \infty)}  $
   and
$|| V_R ||_{L^2[\alpha, \infty)}$ are bounded by constant multiples of   $ e^{- \alpha^{3/2} }$,
and the claim follows.
\end{proof}

\begin{proof}[Proof of Lemma \ref{besselasymp}]
First we consider the scaled $RHP$ for
\[
   M^{(1)}(w) \equiv M^{\sla}( |\alpha| w ),
\]
which has the jump conditions:
\[
   \left(  \ba{ll} 1 & 0 \\ e^{ - |\alpha|^{3/2} ( 2 w^{1/2} - \frac{4}{3} w^{3/2})}  & 1   \ea \right),   \
   \arg w = \pm \frac{2}{3} \pi,
    \  \mbox{ and }     \  \left(  \ba{rr} 0 &1  \\ -1 & 0 \ea \right),   \  w \in \R_{-}.
\]
We will now extract a local parametrix in a neighborhood of $w = 0$.  For
$w \in U_{\ep} = \{ |w| < \ep \}$ and $\ep < 1$ define
\[
   \eta =   \eta(w) =   w ( 1 -  \frac{2}{3} w)^2
\]
Clearly, $\eta(w)$ takes $U_{\ep}$ in a one-to-one fashion onto a open neighborhood
${\wt U}_{\ep}$ of $\nu = 0$, sending the negative real line to itself and the segments
$\arg w = \pm 2 \pi/3$ onto rays $\gamma^{\pm} \subset {\wt U}_{\ep}$ lying above
and below the real axis.   Extending $\gamma^{\pm}$ to $\infty$  (smoothy) along
straight lines outside of  ${\wt U}_{\ep}$ we have the jump relations
\[
   \left(  \ba{ll} 1 & 0 \\ e^{ -2  |\alpha|^{3/2}  \eta^{1/2}}  & 1   \ea \right),   \
   \eta \in \gamma^{\pm}
    \  \mbox{ and }     \  \left(  \ba{rr} 0 &1  \\ -1 & 0 \ea \right),   \  \eta \in \R_{-}.
\]
This identifies the choice of $\gamma^{\pm}$, and with this choice the above problem
is solved by $P_{B}  (|\alpha|^3 \eta)$.  What is the same, $P_{B}(|\alpha|^3 \eta(w))$
satisfies the jump relations for $M^{(1)}(w)$ restricted to $|w| < \ep$ (in which the
upper and lower contours are pulled back to $\arg w = \pm 2 \pi/3.$

Next we perform a second transformation, setting
\be
    M^{(2)}( w )  = \left\{ \ba{ll}  M^{(1)}(w)  \Bigl( {\mathfrak m}( |\alpha| w) \Bigr)^{-1},  &   |w| > \ep,  \\
                                    M^{(1)}(w)  \Bigl(
                    E_{\alpha}(|\alpha| w)
                      P_{\rm B} ( |\alpha|^3 \eta(w)) \Bigr)^{-1},  & |w| < \ep \ea \right. .
\ee
The definition of  $E_{\alpha}(|\alpha| w)$ is given in (\ref{EEEdef}).
It is analytic in $|w| < \ep$ and
so $E_{\alpha}(|\alpha| w) P_{\rm B}(  |\alpha|^3 \eta(w))$ also shares
jump conditions with $M^{(1)}(w)$ in $|w| < \ep$.

The point is that $M^{(2)}( w )$ satisfies a new $RHP$  with  jump contour consisting of
three pieces: the rays $\Gamma^{\pm} \equiv  \arg w = \pm 2 \pi/3$  $\cap \{ |w| > \ep \} $ and
the $\partial U_{\ep}$ the boundary of the disk of radius $\ep.$

On either of the first set of contours, $\Gamma^{\pm}$, the jump matrix is bounded as in
\beqn
  \Bigl| {\mathfrak m}(  |\alpha| w)
   \left(  \ba{cc}  1 & 0 \\ e^{ - |\alpha|^{3/2} (  w - \frac{2}{3} w^{3/2} ) } & 0  \ea \right)
    ( {\mathfrak m} (|\alpha| w) )^{-1}  \Bigr|
      \le   I + e^{ -  \frac{2}{3} |\alpha|^3 |w|^{3/2} }   \left(  \ba{cc}   1 & 1 \\  \alpha^2 |z|^{1/2} & 1 \ea \right)
\eeqn
which is to say it is $I + \OO(e^{-  C_{\ep} |\alpha|^3 })$ in  $L^2 \cap L^{\infty}(\Gamma^{\pm})$. On
$\partial  U_{\ep}$ the jump matrix is $E_{\alpha}(|\alpha|w)  $ $P_{\rm B} ( |\alpha|^3 \eta(w))$ $
({\mathfrak m}(|\alpha| w))^{-1}$
and we compute: for  $|w| = \ep$ and $|\alpha| \gg 1$,
\beqn
E_{\alpha}(|\alpha| w) \,  P_{\rm B} ( |\alpha|^3 \eta(w)) \, ({\mathfrak m}(|\alpha| w))^{-1} & = &
E_{\alpha}(|\alpha| w) \,  \left(  \frac{ w}{ \alpha^2 \eta(w) } \right)^{{\sigma_3}/4}   e^{\frac{ \pi i}{4} \sigma_3} \, \Bigl( I + \OO(\alpha^{-1}) \Bigr) \\
& = & I + \OO(\alpha^{-1}).
\eeqn
It follows that
\[
   M^{(2)}(w)  = I + \OO(\alpha^{-1}),   \mbox{ uniformly for } w  \mbox{ supported away from  }  \partial U_{\ep} \cup \Gamma^{+} \cup \Gamma^{-}.
\]
Undoing the transformations inside of $U_{\ep}$ establishes the claim for $M^{\sla}(\zeta)$.
\end{proof}

\bigskip
\noindent
{\bf Acknowledgments. }  
We are grateful to Arno Kuijlaars for pointing out reference \cite{ClaeysArno} as well as an important glitch
in an earlier version of this paper.
The work of Brian Rider was supported in
part by NSF grant DMS-0505680; Xin Zhou was supported in part by
NSF grant DMS-0602344.

\end{document}